\documentclass[10pt, oneside, leqno]{article}
\usepackage[svgnames]{xcolor}
\usepackage{enumerate}
\usepackage{tikz}
\usepackage{tikz-cd}
\usepackage{amsmath}
\usepackage{amsfonts}
\usepackage{amssymb}
\usepackage{amsthm}
\usepackage{mathrsfs}
\usepackage{mathtools}
\usepackage{color}
\usepackage{graphicx}
\usepackage{pinlabel}  %%% the recommended graphics+labelling package
\usepackage{graphicx}  %%% the recommended graphics package
\usepackage[all]{xy}
\usepackage{amscd}
\usepackage{authblk}
\usepackage{bm}
\usepackage{setspace}
\usetikzlibrary{shapes.geometric, arrows, calc, positioning, decorations.markings, patterns}
\usepackage{here}
\usepackage{hyperref}
\usepackage{comment}
\usepackage{newtxtext,newtxmath}
\newtheorem{thm}{Theorem}[section]    % Standard theorem environment
\newtheorem{lem}[thm]{Lemma}          % Lemma environment with numbering 
\newtheorem{prop}[thm]{Proposition}
\newtheorem{axiom}{Axiom}

\theoremstyle{definition}
\newtheorem{defn}[thm]{Definition}    % Definition environment with 
\newtheorem{rem}[thm]{Remark}             % Unnumbered environment for remarks.
\newtheorem{exa}[thm]{Example}

\theoremstyle{mytheoremstyle}
\newtheorem{cor}[thm]{Corollary}
\numberwithin{equation}{section}

\newcommand*{\longhookrightarrow}{\ensuremath{\lhook\joinrel\relbar\joinrel\rightarrow}}

\def\abb{{\mathbb{A}}}

\def\cbb{{\mathbb{C}}}

\def\fbb{{\mathbb{F}}}

\def\qbb{{\mathbb{Q}}}
\def\rbb{{\mathbb{R}}}

\def\zbb{{\mathbb{Z}}}

\def\ccal{{\mathcal{C}}}
\def\dcal{{\mathcal{D}}}

\def\fcal{{\mathcal{F}}}

\def\ical{{\mathcal{I}}}
\def\jcal{{\mathcal{J}}}

\def\ocal{{\mathcal{O}}}
\def\pcal{{\mathcal{P}}}

\def\scal{{\mathcal{S}}}
\def\tcal{{\mathcal{T}}}

\def\vcal{{\mathcal{V}}}

\def\xfra{{\mathfrak{X}}}

\def\czero{{\mathcal{C}^0}}

\numberwithin{equation}{section}

\usepackage[top=30truemm,bottom=30truemm,left=30truemm,right=30truemm]{geometry}
\newcommand{\bvec}[1]{\mbox{\boldmath $#1$}}

\usepackage{authblk}
\usepackage{caption}
\begin{document}

%%%%% To ease editing, for IMPAN journals add:

\baselineskip=15pt %17pt
\abovedisplayskip=1ex
\belowdisplayskip=1ex
%%%%%%%%%%%%%%%%

\captionsetup[figure]{labelformat={default},labelsep=colon,name={Fig.}}

\title{Smooth singular complexes and \\ diffeological principal bundles}

\author{Hiroshi Kihara}

\affil{}
\date{}
\maketitle

%% Classification and key words; note that the 2010 classification is used:

\renewcommand{\thefootnote}{}

\footnote{ \emph{Date}: November 8, 2021}
\footnote{Center for Mathematical Sciences, University of Aizu, 
	Tsuruga, Ikki-machi, Aizu-Wakamatsu City, Fukushima, 965-8580, Japan} \footnote{\textbf{e-mail: kihara@u-aizu.ac.jp}}

\renewcommand{\thefootnote}{\arabic{footnote}}
\setcounter{footnote}{0}
\vspace{-15ex}

%%%%%%%%
\begin{abstract}
\noindent \textbf{Abstract}. In previous papers, we used the standard simplices $\Delta^p$ $(p\ge 0)$ endowed with diffeologies having several good properties to introduce the singular complex $S^\dcal(X)$ of a diffeological space $X$. On the other hand, Hector and Christensen-Wu used the standard simplices $\Delta^p_{\rm sub}$ $(p\ge 0)$ endowed with the sub-diffeology of $\rbb^{p+1}$ and the standard affine $p$-spaces $\abb^p$ $(p\ge 0)$ to introduce the singular complexes $S^\dcal_{\rm sub}(X)$ and $S^\dcal_{\rm aff}(X)$, respectively, of a diffeological space $X$. In this paper, we prove that $S^\dcal(X)$ is a fibrant approximation both of $S^\dcal_{\rm sub}(X)$ and $S^\dcal_{\rm aff}(X)$. This result easily implies that the homotopy groups of $S^\dcal_{\rm sub}(X)$ and $S^\dcal_{\rm aff}(X)$ are isomorphic to the smooth homotopy groups of $X$, proving a conjecture of Christensen and Wu. Further, we characterize diffeological principal bundles (i.e., principal bundles in the sense of Iglesias-Zemmour) using the singular functor $S^\dcal_{\rm aff}$. By using these results, we extend characteristic classes for $\dcal$-numerable principal bundles to characteristic classes for diffeological principal bundles.
\end{abstract}

\section{Introduction}\label{section1}

In \cite{origin}, we constructed good diffeologies on $\Delta^p=\{(x_0,\ldots,x_p)\in \rbb^{p+1}\:|\: \sum x_i = 1,\: x_i \ge 0 \text{ for any } i \}$ $(p \ge 0)$, and used them to define the singular complex $S^\dcal(X)$ of a diffeological space $X$, introducing a model structure on the category $\dcal$ of diffeological spaces. Further, in \cite{smh}, we also used the singular functor $S^\dcal$ to introduce a simplicial category structure on $\dcal$, and developed a smooth homotopy theory based on the simplicial and model category structures on $\dcal$.

On the other hand, in his study \cite{H} of diffeological spaces, Hector used the sets $\Delta^p$ endowed with the sub-diffeology of $\rbb^{p+1}$ $(p\ge 0)$ to define the singular complex $S^\dcal_{\rm sub}(X)$ of a diffeological space $X$. His singular complex is also used in \cite{Kuri20}. Further, in their attempt of constructing a model structure on $\dcal$ \cite{CW}, Christensen and Wu used the affine spaces $\abb^p=\{(x_0,\ldots,x_p)\in \rbb^{p+1}\:|\: \sum x_i = 1\}$ endowed with the sub-diffeology of $\rbb^{p+1}$ $(p\ge 0)$ to define the singular complex $S^\dcal_{\rm aff}(X)$. Their singular complex is also used in \cite{Kuri20, Kuri20', Bunk}.

As can be seen in the references cited above, the singular complexes $S^\dcal(X)$, $S^\dcal_{\rm sub}(X)$, and $S^\dcal_{\rm aff}(X)$ have played crucial roles in the smooth homotopical study of diffeological spaces. However, the natural weak equivalences between them have not yet been established.

In this paper, we show that the singular complexes $S^\dcal(X)$, $S^\dcal_{\rm sub}(X)$, and $S^\dcal_{\rm aff}(X)$ are weakly equivalent (Theorem \ref{thm1.1}). As a corollary of this result, we identify the homotopy groups of $S^\dcal_{\rm aff}(X)$ and $S^\dcal_{\rm sub}(X)$ with the smooth homotopy groups of $X$, proving a conjecture of Christensen and Wu (Corollary \ref{cor1.2}). Though we mainly use the singular functor $S^{\dcal}$, we also use the singular functor $S^\dcal_{\rm aff}$ to characterize diffeological principal bundles (i.e., principal bundles in the sense of Iglesias-Zemmour) (Theorem \ref{bdletriviality}). This theorem along with the weak equivalence between $S^\dcal_{\rm aff}(X)$ and $S^\dcal(X)$, is used to extend characteristic classes for $\dcal$-numerable principal $G$-bundles to diffeological principal $G$-bundles (Proposition \ref{cor1.3}).

Throughout this paper, $\dcal$ and $\scal$ denote the category of diffeological spaces and the category of simplicial sets, respectively. (See \cite{GJ, May, K} for the basics of simplicial homotopy theory.)

\subsection*{Weak equivalences between $S^\dcal(X)$, $S^\dcal_{\rm sub}(X)$, and $S^\dcal_{\rm aff}(X)$}
The following theorem is the main result of this paper. Note that the canonical maps $\Delta^p \overset{id}{\longrightarrow} \Delta^p_{\rm sub} \longhookrightarrow \abb^p$ $(p\ge 0)$ induce natural morphisms of simplicial sets $S^\dcal_{\rm aff}(X)\longrightarrow S^\dcal_{\rm sub}(X)\longhookrightarrow S^\dcal(X)$ (see Lemma \ref{simplex}(3) and Proposition \ref{natural}). Recall that $S^\dcal(X)$ is always Kan (i.e., fibrant in the category $\scal$); see Corollary \ref{lbryhorn}(1). (Cf. Remark \ref{extra}\textit{}(2).)

\begin{thm}\label{thm1.1}
 The natural morphisms of simplicial sets
 \[
  S^{\dcal}_{\mathrm{aff}}(X)\longrightarrow S^{\dcal}_{\mathrm{sub}}(X) \longhookrightarrow S^{\dcal}(X)
 \]
 are weak equivalences. In particular, $S^{\dcal}(X)$ is a fibrant approximation both of $S^{\dcal}_{\text{\rm aff}}(X)$ and $S^{\dcal}_{\text{\rm sub}}(X)$.
\end{thm}\noindent
The statement that $S^\dcal(X)$ is a fibrant approximation of $S^\dcal_{\rm sub}(X)$ was announced in \cite[Remark A.5]{origin}.\par
%Let $\scal_\ast$ denote the category of pointed simplicial sets, and let $\scal_{\ast f}$ denote the full subcategory of $\scal_\ast$ consisting of fibrant objects (i.e., pointed Kan complexes). We can extend the homotopy group functor $\pi_i: \scal_{\ast f}\to Gr$ to $\scal_{\ast}$ as follows. Choose a fibrant approximation functor $R:\scal_\ast \longrightarrow \scal_{\ast f}$ and define the homotopy group $\pi_i(K,k_0)$ of a pointed simplicial set $(K, k_0)$ by
%\[
%	\pi_i(K, k_0) = \pi_i R(K,k_0).
%\]
Next, we recall that $\pi_i(S^\dcal(X),x)$ is isomorphic to the smooth homotopy group $\pi_i^{\dcal}(X,x)$ (Theorem \ref{homotopygp}), and use Theorem \ref{thm1.1} to identify the homotopy groups of $S^\dcal_{\rm aff}(X)$ and $S^\dcal_{\rm sub}(X)$; see Section 4.4 for the homotopy groups of a simplicial set which need not be Kan. 
\begin{cor}\label{cor1.2}
 Let $(X,x)$ be a pointed diffeological space. Then, both $\pi_i(S^{\dcal}_{\mathrm{aff}}(X), x)$ and $\pi_i(S^{\dcal}_{\mathrm{sub}}(X),x)$ are naturally isomorphic to the smooth homotopy group $\pi_i^{\dcal}(X,x)$ for $i \geq 0$.
\end{cor}\noindent
Christensen and Wu showed that if $S^\dcal_{\rm aff}(X)$ is fibrant, then $\pi_i(S^\dcal_{\rm aff}(X), x)$ is isomorphic to the smooth homotopy group $\pi^\dcal_i(X,x)$ for $i\ge 0$ (\cite[Theorem 4.11]{CW}), and conjectured that for every diffeological space $X$, $\pi_i(S^\dcal_{\rm aff}(X),x)$ is isomorphic to $\pi_i^\dcal(X,x)$ for $i\ge 0$ (\cite[p. 1272]{CW}). Corollary \ref{cor1.2} contains their conjecture.

\subsubsection*{(Co)homology of diffeological spaces}
According to \cite[Section 3.1]{smh}, we define the homology $H_\ast(X;\: A)$ and the cohomology $H^\ast (X;\: A)$ of  a diffeological space $X$ with coefficients in an abelian group $A$ by
\begin{eqnarray*}
	H_\ast(X;\: A) = H_\ast(\zbb S^\dcal(X)\otimes A),\\
        H^\ast (X;\: A) = H^\ast {\rm Hom}(\zbb S^\dcal (X), A),
\end{eqnarray*}
where the simplicial abelian group $\zbb K$ freely generated by a simplicial set $K$ is regarded as a chain complex by setting $\partial=\sum(-1)^i d_i$ (see \cite[Section 3.1]{smh}). It follows from Theorem \ref{thm1.1} that the (co)homology of $X$ is naturally isomorphic to the (co)homologies defined using $S^\dcal_{\rm sub}(X)$ and $S^\dcal_{\rm aff}(X)$ instead of $S^\dcal(X)$. However, this fact is actually proved in Section 3.2 as a key to proving Theorem \ref{thm1.1}; the (co)homology of $X$ is also naturally isomorphic to the cubic (co)homology introduced in \cite[p. 176-186]{IZ} (Remark \ref{cubic}).

\subsection*{Application to diffeological principal bundles}

Let $G$ be a diffeological group. A $\dcal$-numerable principal $G$-bundle $\pi: P\to X$ is a principal $G$-bundle which admits a trivialization open cover $\{U_i\}$ of $X$ and a smooth partition of unity subordinate to it. On the other hand, Iglesias-Zemmour introduced a weaker notion of a principal $G$-bundle; such a principal $G$-bundle, referred to as a diffeological principal $G$-bundle, is defined by local triviality of the pullback along any plot (Definition \ref{diffbdle}(2)).

Though we mainly use the singular complexes $S^\dcal(X)$ in smooth homotopy theory, the singular complexes $S^\dcal_{\rm aff}(X)$, along with Theorem \ref{thm1.1} play an essential role in the study of diffeological principal bundles, as explained below.

\subsubsection*{Characterization of diffeological principal $G$-bundles}
Let $\ccal$ be a category with finite products, and $G$ a group in $\ccal$. Then, $\ccal G$ denotes the category of right $G$-objects of $\ccal$ (i.e., objects of $\ccal$ endowed with a right $G$-action). For $B\in \ccal$, $\ccal G/B$ denotes the category of objects of $\ccal G$ over $B$, where $B$ is regarded as an object of $\ccal G$ with trivial $G$-action.

Since $S^\dcal_{\rm aff}: \dcal \longrightarrow \scal$ is a right adjoint (see Remark \ref{extra}(1)), $S^\dcal_{\rm aff}$ induces the functor $\dcal G/ X$ to $\scal S^\dcal_{\rm aff}(G)/ S^\dcal_{\rm aff}(X)$. We then have the following characterization theorem for diffeological principal $G$-bundles (see Definition \ref{simplicialbdle} for the notion of a simplicial principal bundle).

\begin{thm}\label{bdletriviality}
	\begin{itemize}
		\item[{\rm (1)}] Let $\pi:P\longrightarrow X$ be an object of $\dcal G/X$. Then, $\pi:P\longrightarrow X$ is a diffeological principal $G$-bundle if and only if $S^\dcal_{\rm aff}(\pi): S^\dcal_{\rm aff} (P)\longrightarrow S^\dcal_{\rm aff}(X)$ is a principal $S^\dcal_{\rm aff}(G)$-bundle.
		\item[{\rm (2)}] The functor $S^\dcal_{\rm aff}:\dcal \longrightarrow \scal$ induces a faithful functor from the category $\mathsf{P}\dcal G_{\rm diff}$ of diffeological principal $G$-bundles to the category $\mathsf{P} \scal S^\dcal_{\rm aff}(G)$ of principal $S^\dcal_{\rm aff}(G)$-bundles.
	\end{itemize}
\end{thm}
The essential reason why $S^\dcal_{\rm aff}$ is useful in the study of diffeological principal $G$-bundles is because $S^\dcal_{\rm aff}(X)$ can be regarded as the set of global plots of $X$. We can use Theorem \ref{bdletriviality} to calculate the (co)homology of pathological diffeological spaces such as irrational tori and $\rbb/\qbb$ (see Section 2.3 and Example \ref{R/Q}); other cohomology theories of irrational tori were calculated by Iglesias-Zemmour and Kuribayashi (see Remark \ref{Tgamma}).

\subsubsection*{Characteristic classes of diffeological principal $G$-bundles}
We apply Theorem \ref{bdletriviality} to construct characteristic classes for diffeological principal $G$-bundles.

\begin{comment}
For this, let us recall the following from \cite[Section 3.1]{smh}: The cohomology $H^\ast (X;\: A)$ of  a diffeological space $X$ with coefficients in an abelian group $A$ is defined by
\[
	H^\ast (X;\: A) = H^\ast {\rm Hom}(\zbb S^\dcal (X);\: A),
\]
where the simplicial abelian group $\zbb K$ freely generated by a simplicial set $K$ is regarded as a chain complex by setting $\partial=\Sigma(-1)^i d_i$ (see \cite[Section 3.1]{smh}). (The (co)homology of $X$ is naturally isomorphic not only to the (co)homologies defined using $S^\dcal_{\rm sub}(X)$ and $S^\dcal_{\rm aff}(X)$ instead of $S^\dcal(X)$ but also to the cubic (co)homology introduced in \cite[]{IZ}; see Section 3).
\end{comment}
 A {\it characteristic class}  for a class $\pcal$ of smooth principal $G$-bundles is a rule assigning to a principal $G$-bundle $\pi: P\to X$ in $\pcal$ a cohomology class $\alpha(P)$ of $X$ such that $\alpha(f^{\ast} P) = f^{\ast} \alpha(P)$ holds. Christensen and Wu constructed the universal $\dcal$-numerable principal $G$-bundle $\pi_G: EG\to BG$ and proved that the set of isomorphism classes of $\dcal$-numerable principal $G$-bundles over $X$ bijectively corresponds to the smooth homotopy set $[X,BG]_\dcal$ (\cite[Theorem 5.10]{CW17}). Thus, a cohomology class $\alpha \in H^k(BG;A)$ defines the characteristic class $\alpha(\cdot)$ for the class of $\dcal$-numerable principal $G$-bundles. More precisely, the characteristic class $\alpha(P)\in H^k(X;A)$ of a $\dcal$-numerable principal $G$-bundle $\pi: P\to X$ is defined by
\[
\alpha(P) = f^\ast_P \alpha,
\]
where $f_P: X\longrightarrow BG$ is a classifying map of $P$.

We would like to extend the characteristic class $\alpha(\cdot)$ to the class of diffeological principal $G$-bundles. Since pullbacks of $EG$ are necessarily $\dcal$-numerable, the above definition of the characteristic class $\alpha(\cdot)$ does not apply to the class of diffeological principal $G$-bundles. Further, since the class of diffeological principal $G$-bundles does not have the homotopy invariance property with respect to pullback, it has no classifying space (see \cite[Section 3]{CW17}).

Nevertheless, we can prove the following result.
 \begin{prop}\label{cor1.3}
	Let $G$ be a diffeological group and $\alpha$ an element of $H^k(BG;A)$. Then, the characteristic class $\alpha(\cdot)$ for $\dcal$-numerable principal $G$-bundles extends to a characteristic class for diffeological principal $G$-bundles.
\end{prop}
This paper is organized as follows. In Section 2, we recall the basic notions and results on diffeological spaces and the singular functor $S^\dcal$. In Section 3, we make a brief review on the singular functors $S^\dcal_{\rm sub}$ and $S^\dcal_{\rm aff}$, and show that there exist natural morphisms between $S^\dcal_{\rm aff}(X)$, $S^\dcal_{\rm sub}(X)$, and $S^\dcal(X)$ which induce isomorphisms on (co)homology. Then, we prove Theorem \ref{thm1.1} and Corollary \ref{cor1.2} in Section 4. In Section 5, we recall the notions of a diffeological principal bundle and a simplicial principal bundle, and prove Theorem \ref{bdletriviality}. In Section 6, we prove Proposition \ref{cor1.3} and discuss the sets of characteristic classes for the three classes $\mathsf{P}\dcal G$, $\mathsf{P}\dcal G_{\rm num}$, and $\mathsf{P}\dcal G_{\rm diff}$ of smooth principal $G$-bundles (see Definition \ref{diffbdle}(3) for these three classes).

\section{Diffeological spaces}
In this section, we first recall the convenient properties of the category $\dcal$ of diffeological spaces, along with the adjoint pair $\widetilde{\cdot}: \dcal \rightleftarrows \czero: R$ of the underlying topological space functor and its right adjoint (Section 2.1). Then, we recall the standard simplices $\Delta^p$ $(p\ge 0)$ and the adjoint pair $|\ |_\dcal: \scal \rightleftarrows \dcal:S^\dcal$ of the realization and singular functors (see Section 2.2). Last, we make a brief review of some results of \cite{smh}, in which the adjoint pairs $(\widetilde{\cdot}, R)$ and $(|\ |_\dcal, S^\dcal)$ play an essential role (Section 2.3).

\subsection{Categories $\dcal$ and $\czero$}
In this subsection, we summarize the convenient properties of the category $\dcal$ of diffeological spaces, recalling the adjoint pair $\widetilde{\cdot}: \dcal \rightleftarrows \czero: R$ of the underlying topological space functor and its right adjoint; see \cite{IZ} and \cite{origin} for full details.
\par\indent
Let us begin with the definition of a diffeological space. A {\sl parametrization} of a set $X$ is a (set-theoretic) map $p: U \longrightarrow X$, where $U$ is an open subset of $\rbb^{n}$ for some $n$.
\begin{defn}\label{diffeological}
	\begin{itemize}
		\item[(1)] A {\sl diffeological space} is a set $X$ together with a specified set $D_X$ of parametrizations of $X$ satisfying the following conditions:
		\begin{itemize}
			\item[(i)](Covering)  Every constant parametrization $p:U\longrightarrow X$ is in $D_X$.
			\item[(ii)](Locality) Let $p :U\longrightarrow X$ be a parametrization such that there exists an open cover $\{U_i\}$ of $U$ satisfying $p|_{U_i}\in D_X$. Then, $p$ is in $D_X$.
			\item[(iii)](Smooth compatibility) Let $p:U\longrightarrow X$ be in $D_X$. Then, for every $n \geq 0$, every open set $V$ of $\rbb^{n}$ and every smooth map $F  :V\longrightarrow U$, $p\circ F$ is in $D_X$.
		\end{itemize}
		The set $D_X$ is called the {\sl diffeology} of $X$, and its elements are called {\sl plots}.
		\item[(2)] Let $X=(X,D_X)$ and $Y=(Y,D_Y)$ be diffeological spaces, and let $f  :X\longrightarrow Y$ be a (set-theoretic) map. We say that $f$ is {\sl smooth} if for any $p\in D_X$, \ $f\circ p\in D_Y$.
	\end{itemize}
\end{defn}
The convenient properties of $\dcal$ are summarized in the following proposition. Recall that a topological space $X$ is called {\sl arc-generated} if its topology is final for the continuous curves from $\rbb$ to $X$, and let $\czero$ denote the category of arc-generated spaces and continuous maps. See \cite[pp. 230-233]{FK} for initial and final structures with respect to the underlying set functor.
\begin{prop}\label{conven}
	\begin{itemize}
		\item[$(1)$] The category ${\dcal}$ has initial and final structures with respect to the underlying set functor. In particular, ${\dcal}$ is complete and cocomplete.
		\item[$(2)$] The category $\mathcal{D}$ is cartesian closed.
		\item[$(3)$] The underlying set functor $\dcal \longrightarrow Set$
		is factored as the underlying topological space functor
		$\widetilde{\cdot}:\dcal \longrightarrow \czero$
		followed by the underlying set functor
		$\czero \longrightarrow Set$.
		Further, the functor
		$\widetilde{\cdot}:\dcal \longrightarrow \czero$
		has a right adjoint
		$R:\czero \longrightarrow \dcal$.
	\end{itemize}
	\begin{proof}
		See \cite[p. 90]{CSW}, \cite[pp. 35-36]{IZ}, and \cite[Propositions 2.1 and 2.10]{origin}.
	\end{proof}
\end{prop}
The following remark relates to Proposition \ref{conven}.
\begin{rem}\label{convenrem}
	\begin{itemize}
		\item[${\rm (1)}$] Let $\xfra$ be a concrete category (i.e., a category equipped with a faithful functor to $Set$); the faithful functor $\xfra \longrightarrow Set$ is called the underlying set functor. See \cite[Section 8.8]{FK} for the notions of an $\xfra$-embedding, an $\xfra$-subspace, an $\xfra$-quotient map, and an $\xfra$-quotient space. $\dcal$-subspaces and $\dcal$-quotient spaces are usually called diffeological subspaces and quotient diffeological spaces, respectively.
		\item[${\rm (2)}$] For Proposition \ref{conven}(3), recall that the underlying topological space $\tilde{A}$ of a diffeological space $A = (A, D_A)$ is defined to be the set $A$ endowed with the final topology for $D_A$ and that $R$ assigns to an arc-generated space $X$ the set $X$ endowed with the diffeology
		$$
		D_{RX} = \text{\{continuous parametrizations of } X \text{\}.}
		$$
		Then, we can easily see that $\tilde{\cdot} \circ R = Id_{\ccal^0}$ and that the unit $A \longrightarrow R\widetilde{A}$ of the adjoint pair $(\widetilde{\cdot}, R)$ is set-theoretically the identity map.
		\item[${\rm (3)}$] The notion of an arc-generated space is equivalent to that of a $\Delta$-generated space (see \cite{CSW}, \cite[Section 2.2]{origin}). The categories $\dcal$ and $\czero$ share convenient properties (1) and (2) in Proposition \ref{conven}, which often enables us to deal with $\dcal$ and $\czero$ simultaneously (see \cite{smh}). See \cite[Remark 2.4]{smh} for the reason why $\czero$ is the most suitable category as a target category of the underlying topological space functor for diffeological spaces.
	\end{itemize}
\end{rem}
\subsection{Standard simplices $\Delta^p$}
In this subsection, we recall the standard simplices $\Delta^p$ $(p\ge 0)$, along with the adjoint pair $|\ |_\dcal: \scal \rightleftarrows \dcal:S^\dcal$ of the realization and singular functors.

In \cite{origin}, we introduced a model structure on the category $\dcal$. The principal part of our construction of a model structure on $\dcal$ is the construction of good diffeologies on the sets
\[
	\Delta^p = \{(x_0,\ldots, x_p)\in \rbb^{p+1} \:|\: \sum x_i =1, \ x_i \ge 0 \text{ for any } i\} \ \ \ (p \ge 0)
\]
which enable us to define weak equivalences, fibrations, and cofibrations and to verify the model axioms (see Remark \ref{modelstr}). The required properties of the diffeologies on $\Delta^{p} \ (p \geq 0)$ are expressed in the following four axioms:
\begin{axiom}
	The underlying topological space of $\Delta^p$ is the topological standard $p$-simplex $\Delta^{p}_{\rm top}$ for $p\geq 0$.
\end{axiom}
Recall that
$f:\Delta^p \longrightarrow \Delta^q$
is an {\sl affine map} if $f$ preserves convex combinations.
\begin{axiom}
	Any affine map $f:\Delta^p\longrightarrow \Delta^q$ is smooth.
\end{axiom}
For $K \in \scal$, the {\sl simplex category} $\Delta\downarrow K$ is defined to be the full subcategory of the overcategory $\scal \downarrow K$ consisting of maps $\sigma : \Delta[n] \rightarrow K$. By Axiom 2, we can consider the diagram $\Delta\downarrow K \longrightarrow \mathcal{D}$ sending $\sigma :\Delta[n] \longrightarrow K$ to $\Delta^n$. Thus, we define the {\sl realization functor}
$$
|\ |_{\dcal}: \mathcal{S}\longrightarrow \mathcal{D}
$$
by $|K|_{\mathcal{D}}= \underset{\mathrm{\Delta\downarrow} K}{\mathrm{colim}} \ \Delta^n$.
\par\indent
Consider the smooth map $|\dot{\Delta}[p]|_{\dcal} \longhookrightarrow |\Delta[p]|_{\dcal} = \Delta^{p}$ induced by the inclusion of the boundary $\dot{\Delta}[p]$ into $\Delta[p]$.
\begin{axiom}
	The canonical smooth injection
	$$\left| \dot{\Delta}[p] \right|_{\dcal} \longhookrightarrow \Delta^p$$
	is a $\dcal$-embedding.
\end{axiom}
The $\dcal$-homotopical notions, especially the notion of a $\dcal$-deformation retract, are defined in the same manner as in the category of topological spaces by using the unit interval $I=[0,1]$ endowed with a diffeology via the canonical bijection with $\Delta^{1}$ (\cite[Section 2.4]{origin}).
The {\sl $k^{th}$ horn} of $\Delta^p$ is a diffeological subspace of $\Delta^p$ defined by
\begin{equation*}
\Lambda^{p}_{k} =  \{(x_0,\ldots,x_p)\in\Delta^p \ |\ x_i=0 \hbox{ for some }i\neq k\}.
\end{equation*}
\begin{axiom}
	The $k^{th}$ horn $\Lambda^p_k$ is a $\dcal$-deformation retract of $\Delta^p$ for $p \geq 1$ and $0 \leq k \leq p$.
\end{axiom}
For a subset $A$ of the affine $p$-space $\abb^p = \{(x_0, \ldots, x_p) \in \rbb^{p+1}\ |\ \sum x_i = 1 \}$, $A_{\rm sub}$ denotes the set $A$ endowed with the sub-diffeology of $\abb^p$ (and hence of $\rbb^{p+1}$). The diffeological spaces $\Delta^p_{\rm sub}$ $(p\ge0)$ satisfy Axioms 1-2, but $\Delta^{p}_{\mathrm{sub}}$ satisfies neither Axiom 3 nor 4 for $p \geq 2$ (\cite[Proposition A.2]{origin}). Thus, we must construct a new diffeology on $\Delta^p$, at least for $p \geq 2$.
\par\indent
Let $(i)$ denote the vertex $(0, \ldots, \underset{(i)}{1}, \ldots, 0)$ of $\Delta^p$, and let $d^i$ denote the affine map from $\Delta^{p-1}$ to $\Delta^p$, defined by
\begin{equation*}
d^i((k))= \left \{
\begin{array}{ll}
(k) & \text{for} \ k<i,\\
(k+1)& \text{for} \ k\geq i.
\end{array}
\right.
\end{equation*}

\begin{defn}\label{simplices}
	We define the {\sl standard $p$-simplices} $\Delta^p$ ($p\geq 0$) inductively. Set $\Delta^p=\Delta_{\mathrm{sub}}^p$ for $p\leq 1$. Suppose that the diffeologies on $\Delta^k$ ($k<p$) are defined.
	We define the map
	\begin{eqnarray*}
		\varphi_i: \Delta^{p-1}\times [0,1) & \longrightarrow &  \Delta^p
	\end{eqnarray*}
	by $\varphi_{i}(x, t) = (1-t)(i)+td^{i}(x)$, and endow $\Delta^p$ with the final structure for the maps $\varphi_{0}, \ldots, \varphi_{p}$.
\end{defn}
The following result is established in \cite[Propositions 3.2, 5.1, 7.1, and 8.1]{origin}.
\begin{prop}\label{axioms}
	The standard $p$-simplices $\Delta^p\ (p \geq 0)$ in Definition \ref{simplices} satisfy Axioms 1-4.
\end{prop}
Without explicit mention, the symbol $\Delta^p$ denotes the standard $p$-simplex defined in Definition \ref{simplices} and a subset of $\Delta^p$ is endowed with the sub-diffeology of $\Delta^p$. Since the diffeology of $\Delta^p$ is the sub-diffeology of $\abb^{p}$ for $p \leq 1$, the $\dcal$-homotopical notions, especially the notion of a $\dcal$-deformation retract, coincide with the ordinary smooth homotopical notions in the theory of diffeological spaces (see \cite[p. 108]{IZ} and \cite[Remark 2.14]{origin}).

Since $\Delta^\bullet = \{\Delta^p\}$ is a cosimplicial diffeological space by Axiom 2, the singular complex $S^\dcal(X)$ is defined by
\[
S^\dcal(X) = \dcal(\Delta^\bullet, X).
\]
We can easily see that  $|\ |_{\dcal}: \scal \rightleftarrows \dcal: S^{\dcal}$ is an adjoint pair (\cite[Proposition 9.1]{origin}). Further, we can derive the following result from Proposition \ref{axioms}.

\begin{cor}\label{lbryhorn}
	\begin{itemize}
		\item[{\rm (1)}] The natural isomorphisms
		\[
		|\Delta[p]|_\dcal = \Delta^p, \ |\dot{\Delta}[p]|=\dot{\Delta}^p,\ and \ |\Lambda_k [p]|_\dcal = \Lambda^p_k
		\]
		exist.
		\item[{\rm (2)}] $S^\dcal X$ is a Kan complex for any diffeological space $X$.
	\end{itemize}
	\begin{proof}
		\begin{itemize}
			\item[{\rm (1)}] See \cite[Proposition 9.2]{origin}.
			\item[{\rm (2)}] See \cite[Lemma 9.4(1)]{origin}.\qedhere
		\end{itemize}
	\end{proof}
\end{cor}

See \cite[Section 3.1]{CW} or \cite[Chapter 5]{IZ} for the smooth homotopy groups $\pi^{\dcal}_{p}(X, x)$ of a pointed diffeological space $(X, x)$. Note that $S^{\dcal}X$ is always a Kan complex (Corollary \ref{lbryhorn}(2)) and see \cite[p. 25]{GJ} for the homotopy groups $\pi_p(K, x)$ of a pointed Kan complex $(K, x)$.
\begin{thm}\label{homotopygp}
	Let $(X, x)$ be a pointed diffeological space. Then, there exists a natural bijection
	$$
	\varTheta_{X} : \pi^{\dcal}_{p}(X, x) \longrightarrow \pi_{p}(S^{\dcal}X, x) \ \  \text{for} \ \ p \geq 0,
	$$
	that is an isomorphism of groups for $p > 0$.
	\begin{proof}
		See \cite[Theorem 1.4]{origin}.
	\end{proof}
\end{thm}

\begin{rem}\label{modelstr}
	\begin{itemize}
		\item[{\rm (1)}] Define a map $f :X\longrightarrow Y$ in $\mathcal{D}$ to be
		\begin{itemize}
			\item[$({\rm i})$]
			a {\sl weak equivalence} if $S^{\mathcal{D}} f:S^{\mathcal{D}} X\longrightarrow S^{\mathcal{D}} Y$ is a weak equivalence in the category of simplicial sets,
			\item[$({\rm ii})$]
			a {\sl fibration} if the map $f$ has the right lifting property with respect to the inclusions $\Lambda^p_k \longhookrightarrow\Delta^p$ for all $p>0$ and $0\leq k\leq p$, and
			\item[$({\rm iii})$]
			a {\sl cofibration} if the map $f$ has the left lifting property with respect to all maps that are both fibrations and weak equivalences.
		\end{itemize}
		Then, $\mathcal{D}$ is a compactly generated model category whose object is always fibrant. In fact, the sets of morphisms of $\dcal$
		\begin{eqnarray*}
			\ical & = & \{\dot{\Delta}^{p} \longhookrightarrow \Delta^{p} \ | \ p\geq 0 \},\\
			\jcal & = & \{\Lambda^{p}_{k} \longhookrightarrow \Delta^{p} \ | \ p>0,\ 0 \leq k \leq p \}
		\end{eqnarray*}
		are the sets of generating cofibrations and generating trivial cofibrations respectively (\cite[Theorem 1.3]{origin}). See \cite[Definition 15.2.1]{MP} for a compactly generated model category. By Theorem \ref{homotopygp}, weak equivalences in $\dcal$ are just smooth maps inducing isomorphisms on smooth homotopy groups.
		\item[{\rm (2)}] The adjoint pairs
		$$
		|\ |_{\dcal}: \scal \rightleftarrows \dcal: S^{\dcal} \text{ and } \tilde{\cdot}: \dcal \rightleftarrows \ccal^{0}: R
		$$
		are pairs of Quillen equivalences (\cite[Theorem 1.5]{smh}). Note that the composite of these adjoint pairs is just the adjoint pair
		$$
		|\ | : \scal \rightleftarrows \ccal^0 : S
		$$
		of the topological realization and singular functors.
	\end{itemize}
\end{rem}

\subsection{Homotopy type of $S^\dcal(X)$}
In this subsection, we recall from \cite{smh} the basic results on the homotopy type of $S^\dcal(X)$; they are not essential in the later sections, but they are related to a few results in Section 6.

For a diffeological space $X$, consider the unit $id:X\longrightarrow R\widetilde{X}$ of the adjoint pair $\widetilde{\cdot}:\dcal \rightleftarrows \czero:R$. By applying $S^\dcal(=\dcal(\Delta^\bullet, \cdot))$, we have the natural inclusion
\[
S^\dcal X \longhookrightarrow S \widetilde{X}
\]
(see Proposition \ref{axioms}(Axiom 1)).

If $X$ is a good diffeological space such as a cofibrant object or a $C^\infty$-manifold in the sense of \cite[Section 27]{KM}, then $S^\dcal X \longhookrightarrow S\widetilde{X}$ is a weak equivalence (\cite[Corollary 1.6, Proposition 2.6, and Theorem 11.2]{smh}). Hence, we can calculate the homotopy groups and the (co)homology groups of such good diffeological spaces as those of the underlying topological spaces.

On the other hand, if $X$ is a pathological diffeological space such as an irrational torus, then $S^\dcal X\longhookrightarrow S\widetilde{X}$ is not a weak equivalence (see \cite[Appendix A]{smh}). See Section 6.2 for an approach to the homotopy type of $S^\dcal(X)$ of pathological diffeological spaces $X$ such as irrational tori and $\rbb/\qbb$.

\begin{rem}\label{}
The (co)homology and homotopy groups of diffeological spaces have the same desirable properties as those of topological spaces. Further, the (co)homology and homotopy groups of a diffeological space are just those of its singular complex. Thus, we can apply various algebraic topological and simplicial homotopical tools to the calculation of the (co)homology and homotopy groups of a diffeological space $X$ whether or not $X$ is a good diffeological space (see \cite[Section 3.1]{smh}, Theorem \ref{homotopygp}, and Remark \ref{bdletriviality2}).
\end{rem}

\section{Smooth singular complexes}\label{section2}
In this section, we summarize the basic notions and results on the smooth singular complexes $S^\dcal_{\rm sub}(X)$ and $S^\dcal_{\rm aff}(X)$ (Section 3.1), and then show that there exist natural morphisms between $S^\dcal_{\rm aff}(X)$, $S^\dcal_{\rm sub}(X)$, and $S^\dcal(X)$ which induce chain homotopy equivalences, and hence isomorphisms on (co)homology (Section 3.2). We also show that the singular functors $S^\dcal_{\rm aff}$, $S^\dcal_{\rm sub}$, and $S^\dcal$ transform diffeological coverings to simplicial coverings (Section 3.3); this result is used to reduce the proof of Theorem \ref{thm1.1}.

\subsection{Smooth singular complexes $S^\dcal(X)$, $S^\dcal_{\rm sub}(X)$, and $S^\dcal_{\rm aff}(X)$}
By using the cosimplicial diffeological space $\Delta^\bullet = \{\Delta^p\}$, the singular complex $S^\dcal(X)$ is defined by
\[
	S^\dcal(X) = \dcal(\Delta^\bullet, X),
\]
which is intensively studied in \cite{origin, smh} (see Section 2.2).

Let $\abb^p$ denote the affine $p$-space $\{(x_0, \ldots, x_p)\in \rbb^{p+1}\:|\: \sum x_i = 1\}$ endowed with the sub-diffeology of $\rbb^{p+1}$. Since $\abb^\bullet = \{\abb^p \}$ is a cosimplicial diffeological space, the singular complex $S^\dcal_{\rm aff}(X)$ is defined by
\[
S^\dcal_{\rm aff}(X) = \dcal(\abb^\bullet, X).
\]
The singular complex $S^\dcal_{\rm aff}(X)$ was introduced by Christensen-Wu \cite{CW} ; they used the singular functor $S^\dcal_{\rm aff}$ to define the classes of weak equivalences, fibrations, and cofibrations in $\dcal$, but the model axioms are not yet verified.

Let $\Delta^p_{\rm sub}$ denote the set $\Delta^p$ endowed with the sub-diffeology of $\abb^p$. Since $\Delta^\bullet_{\rm sub} = \{\Delta^p_{\rm sub}\}$ is a cosimplicial diffeological space, the singular complex $S^\dcal_{\rm sub}(X)$ is defined by
\[
S^\dcal_{\rm sub}(X) = \dcal(\Delta^\bullet_{\rm sub}, X).
\]
The singular complex $S^\dcal_{\rm sub}(X)$ was used by Hector \cite{H} to study diffeological spaces by homotopical means.

Now, we summarize the basic properties of $\Delta^p$, $\Delta^p_{\rm sub}$, and $\abb^p$, and the relations among them, which are needed later. A subset $A$ of $\abb^p$ endowed with the sub-diffeology of $\abb^p$ is denoted by $A_{\rm sub}$. The notion of $\dcal$-contractibility  (or smooth contractibility) is defined in the obvious manner (a $\dcal$-contractible diffeological space is often called simply a contractible diffeological space if there is no confusion in context).
\begin{lem}\label{simplex}
	\begin{itemize}
		\item[{\rm (1)}] The diffeological spaces $\Delta^p$, $\Delta^p_{\rm sub}$, and $\abb^p$ are smoothly contractible.
		\item[{\rm (2)}] The underlying topological spaces of $\Delta^p$ and $\Delta^p_{\rm sub}$ are just the topological standard $p$-simplex. The underlying topological space of $\abb^p$ is just the set $\abb^p$ endowed with the usual topology.
		\item[{\rm (3)}] The map $id : \Delta^{p} \longrightarrow \Delta^{p}_{\rm sub}$ is smooth, which restricts to the diffeomorphism $id: \Delta^{p}-{\rm sk}_{p-2}\ \Delta^{p} \xrightarrow[\cong]{} (\Delta^{p} - {\rm sk}_{p-2}\ \Delta^{p})_{\rm sub}$, where ${\rm sk}_{p-2}$ $\Delta^{p}$ denotes the $(p-2)$-skeleton of $\Delta^{p}$.
	\end{itemize}
\end{lem}
\begin{proof}
	\begin{itemize}
		\item[{\rm (1)}] The smooth contractibility of $\Delta^p_{\rm sub}$ and $\abb^p$ are obvious. See \cite[Remark 9.3]{origin} for the smooth contractibility of $\Delta^p$.
		\item[{\rm (2)}] The result for $\Delta^p$ follows from Proposition \ref{axioms}. The results for $\Delta^p_{\rm sub}$ and $\abb^p$ follow from \cite[Lemma 2.12]{origin}.
		\item[{\rm (3)}] See \cite[Lemmas 3.1 and 4.2]{origin}.\qedhere
	\end{itemize}
\end{proof}

\begin{rem}\label{extra}
	In this remark, we recall the left adjoints of $S^\dcal_{\rm sub}$ and $S^\dcal_{\rm aff}$, and see that $S^\dcal_{\rm sub}(X)$ and $S^\dcal_{\rm aff}(X)$ need not be Kan.
	\begin{itemize}
		\item[{\rm (1)}] As mentioned above, the realization functor $|\ |_\dcal: \scal \longrightarrow \dcal$ is a left adjoint of the singular functor $S^\dcal:\dcal \longrightarrow \scal$, and the composite of the adjoint pairs $|\ |_\dcal:\scal \rightleftarrows \dcal:S^\dcal$ and $\tilde{\cdot}:\dcal \rightleftarrows \czero:R$ is just the adjoint pair $|\ |: \scal \rightleftarrows \czero:S$ (see Remark \ref{modelstr}(2)).\\
		\hspace{1.0em} Similarly, we can define the realization functor $|\ |'_\dcal: \scal \longrightarrow \dcal$ by
		\[
		|K|'_\dcal = \underset{\Delta \downarrow K}{\rm colim}\: \Delta^n_{\rm sub},
		\]
		which is a left adjoint of the singular functor $S^\dcal_{\rm sub}:\dcal \longrightarrow \scal$. The composite of the adjoint pairs $|\ |'_\dcal:\scal \rightleftarrows \dcal:S^\dcal_{\rm sub}$ and $\tilde{\cdot}:\dcal \rightleftarrows \czero:R$ is also just the adjoint pair $|\ |:\scal \rightleftarrows \czero:S$ (see Lemma \ref{simplex}(2)).\\
		\hspace{1.0em} The realizations $|K|_\dcal$ and $|K|'_\dcal$ of a simplicial complex $K$ viewed as a simplicial set (\cite[Example 1.4]{May}) are just the diffeological polyhedra $|K|_\dcal$ and $|K|'_\dcal$ respectively, in \cite[Section 8.1]{smh}; they played an essential role in the proof of the homotopy cofibrancy theorem \cite[Theorem 1.10]{smh}.\\
		\hspace{1.0em} Christensen and Wu \cite{CW} defined the realization functor $|\ |''_{\dcal}:\scal \longrightarrow \dcal$ by
		\[
		|K|''_\dcal = \underset{\Delta\downarrow K}{\rm colim}\: \abb^n,
		\]
		which is a left adjoint of the singular functor $S^\dcal_{\rm aff}:\dcal \longrightarrow \scal$.
		\item[{\rm (2)}]  Let us see that $S^{\dcal}_{\rm sub}(X)$ need not be Kan. For this, we consider the extension problem in $\scal$
		\begin{center}
			\begin{tikzcd}
			\Lambda_0[2] \arrow[r, "d^1+d^2"] \arrow[d, hook'] & S^\dcal_{\rm sub}(\Lambda^2_{0\:{\rm sub}})\\
			\Delta[2] , \arrow[ur, dashed]
			\end{tikzcd}
		\end{center}
\vspace{-1.5ex} where $\Lambda_0[2] \xrightarrow{d^1+d^2} S^\dcal_{\rm sub}(\Lambda^2_{0\:{\rm sub}})$ is the simplicial map whose restriction to the $i$-th face corresponds to (the corestriction of) $d^i: \Delta^1\longrightarrow \Delta^2$ for $i=1,2$. Suppose that this extension problem has a solution $r$. Then, we have the commutative diagram in $\dcal$ \vspace{-1.5ex} 
		\begin{center}
			\begin{tikzcd}
			\lvert \Lambda_0 [2]\rvert'_\dcal \arrow[r, "d^1+d^2"] \arrow[d, hook'] & \Lambda^2_{0\:{\rm sub}}\\
			\Delta^2_{\rm sub} \arrow[ru, "r" swap]
			\end{tikzcd}
		\end{center}
		\vspace{-1.5ex} (see Part 1). Noticing that $\lvert \Lambda_0 [2]\rvert'_\dcal$ can be set-theoretically identified with $\Lambda^2_0$, we see that $r$ is a $\dcal$-retraction of $\Delta^2_{\rm sub}$ onto $\Lambda^2_{0\;{\rm sub}}$, which is a contradiction (see \cite[Proposition A.2(2)]{origin}); see also \cite[Remark 8.2]{smh}.\\
\hspace{1em}Similarly, we can use \cite[Theorem 5.13]{BJ} to see that $S^\dcal_{\rm aff}((d^1 \abb^1 \cup d^2 \abb^1)_{\rm sub})$ is not Kan; however, it has already been shown that $S^\dcal_{\rm aff}(X)$ need not be Kan (see \cite[Section 4.3]{CW}).
	\end{itemize}
\end{rem}

\subsection{Natural transformations between $S^\dcal_{\rm aff}$, $S^\dcal_{\rm sub}$, and $S^\dcal$}
In this subsection, we construct natural morphisms between $S^\dcal_{\rm aff}(X)$, $S^\dcal_{\rm sub}(X)$, and $S^\dcal(X)$, and show that they induce chain homotopy equivalences between $\zbb S^\dcal_{\rm aff}(X)$, $\zbb S^\dcal_{\rm sub}(X)$, and $\zbb S^\dcal(X)$, and hence isomorphisms on the (co)homology with arbitrary coefficients.

First, we show that the singular functors $S^\dcal$, $S^\dcal_{\rm sub}$, and $S^\dcal_{\rm aff}$ preserve homotopy. See Section 2.2 for the $\dcal$-homotopical notions.
\begin{lem}\label{homotopypreserving}
	For smooth maps $f,g: X\longrightarrow Y$, consider the following conditions:
	\begin{itemize}
		\item[{\rm (i)}] $f \simeq_\dcal g: X\longrightarrow Y$.
		\item[{\rm (ii)}] $S^\dcal f \simeq S^\dcal g: S^\dcal(X) \longrightarrow S^\dcal (Y)$.
		\item[{\rm (iii)}] $H_\ast(f) = H_\ast(g): H_\ast(X)\longrightarrow H_\ast(Y)$.
	\end{itemize}
	Then, the implications ${\rm (i)}\Rightarrow {\rm (ii)}\Rightarrow {\rm (iii)}$ hold. The same conclusion applies to the functors $S^\dcal_{\rm sub}$ and $S^\dcal_{\rm aff}$, and their homologies.
	\begin{proof}
		{\sl The result for $S^\dcal$.} See \cite[Lemma 9.4(2)]{origin} for ${\rm (i)}\Rightarrow {\rm (ii)}$. For ${\rm (ii)\Rightarrow {\rm (iii)}}$, see \cite[pp. 12-13]{May}.
		
\hspace{-0.6cm}{\sl The result for $S^\dcal_{\rm sub}$}. Recall that $\Delta^1 = \Delta^1_{\rm sub}$. Then, a similar argument applies.
		
\hspace{-0.6cm}{\sl The result for $S^\dcal_{\rm aff}$}. Observe that $f\simeq_\dcal g$ if and only if there exists a smooth map $H:X \times \abb^1 \longrightarrow Y$ such that $H(\cdot, (0))=f$ and $H(\cdot, (1))=g$. Then, a similar argument applies.
	\end{proof}
\end{lem}

Using Lemmas \ref{simplex} and \ref{homotopypreserving}, we can prove the following result.

\begin{prop}\label{natural}
	There exist natural morphisms of simplicial sets
	\[
	S^\dcal_{\rm aff}(X) \longrightarrow S^\dcal_{\rm sub}(X) \longhookrightarrow S^\dcal(X)
	\]
	which induce chain homotopy equivalences
	\[
		\zbb S^\dcal_{\rm aff}(X)\longrightarrow \zbb S^\dcal_{\rm sub}(X)\longrightarrow \zbb S^\dcal(X).
	\]
	\begin{proof}
		We prove the result in three steps. \vspace{1ex}\\
Step 1. {\sl Construction of natural morphisms}. By Lemma \ref{simplex}(3), we have the canonical morphisms of cosimplicial diffeological spaces
			\[
			\Delta^\bullet \overset{id}{\longrightarrow} \Delta^\bullet_{\rm sub} \longhookrightarrow \abb^\bullet,
			\]
			which induce natural morphisms 
			\[
			S^\dcal_{\rm aff}(X) \overset{\kappa}{\longrightarrow} S^\dcal_{\rm sub}(X) \overset{\iota}{\longhookrightarrow} S^\dcal(X).
			\]
			Step 2. We show that for $p\ge 0$, the following hold:
			\[
				H_\ast \zbb S^\dcal_{\rm aff}(\Delta^p) \cong H_\ast \zbb S^\dcal_{\rm sub}(\Delta^p) \cong H_\ast \zbb S^{\dcal} (\Delta^p)\cong \zbb[0],
			\]
			where $\zbb[0]$ denotes the graded module with $\zbb[0]_0 = \zbb$ and $\zbb[0]_i = 0$ $(i\neq 0)$. It is easily seen that these isomorphisms hold for $p=0$. Thus, they hold for any $p\ge 0$ by Lemmas \ref{homotopypreserving} and \ref{simplex}(1).\vspace{2mm}\\
			Step 3. To prove the rest of the statement, we ``augment" the singular chain complexes $\zbb S^\dcal(X)$, $\zbb S^\dcal_{\rm sub}(X)$, and $\zbb S^\dcal_{\rm aff}(X)$ in a canonical manner (see \cite[p. 194]{EM}); the augmented singular chain complexes are denoted by $\zbb S^\dcal(X)\tilde{}$, $\zbb S^\dcal_{\rm sub}(X)\tilde{}$, and $\zbb S^\dcal_{\rm aff}(X)\tilde{}$. Then, we have
			\[
				H_\ast \zbb S^\dcal_{\rm aff}(\Delta^p)\tilde{} = H_\ast \zbb S^\dcal_{\rm sub}(\Delta^p)\tilde{} = H_\ast \zbb S^\dcal(\Delta^p)\tilde{} = 0
			\]
			(Step 2). Since each component of degree $\geq 0$ of $\zbb S^\dcal(X)\tilde{}$ (resp., $\zbb S^\dcal_{\rm sub}(X)\tilde{}$, $\zbb S^\dcal_{\rm aff}(X)\tilde{}$) is representable for the set of model objects $\{\Delta^p\}_{p\ge 0}$ (resp., $\{\Delta^p_{\rm sub}\}_{\rm p\ge 0}$, $\{\abb^p\}_{p\ge 0}$) in the sense of \cite[p. 189]{EM}, we can use \cite[Theorem II]{EM} to construct chain homotopy inverses of the augmented natural chain maps
			\[
				\zbb S^\dcal_{\rm aff}(X)\tilde{} \overset{\zbb \kappa \tilde{}}{\longrightarrow} \zbb S^\dcal_{\rm sub}(X)\tilde{} \overset{\zbb \iota\tilde{}}{\longrightarrow} \zbb S^\dcal(X)\tilde{}
			\]
			so that they restrict to chain homotopy inverses of the natural chain maps
			\[
			\zbb S^\dcal_{\rm aff}(X) \overset{\zbb \kappa}{\longrightarrow} \zbb S^\dcal_{\rm sub}(X)\overset{\zbb \iota}{\longrightarrow} \zbb S^\dcal(X)
			\]
			(see Step 1).
	\end{proof}
\end{prop}

Recall the definitions of $H_{\ast}(X;A)$ and $H^{\ast}(X;A)$ from Section 1.

\begin{cor}\label{quism}
	Let $A$ be an abelian group.
	\begin{itemize}
		\item[{\rm (1)}] The natural morphisms of simplicial sets
		\[
			S^\dcal_{\rm aff}(X)\longrightarrow S^\dcal_{\rm sub}(X) \longhookrightarrow S^\dcal(X)
		\]
		induce isomorphisms of graded modules
		\[
			H_\ast (\zbb S^\dcal_{\rm aff}(X)\otimes A) \underset{\cong}{\longrightarrow} H_\ast(\zbb S^\dcal_{\rm sub}(X)\otimes A) \underset{\cong}{\longrightarrow} H_\ast(X; A),
		\]
		\[
			H^\ast{\rm Hom}(\zbb S^\dcal_{\rm aff}(X), A) \underset{\cong}{\longleftarrow} H^\ast {\rm Hom}(\zbb S^\dcal_{\rm sub}(X), A) \underset{\cong}{\longleftarrow} H^\ast(X; A).
		\]
		\item[{\rm (2)}] If $A$ is a commutative associative ring with unit, then $H^\ast(X; A)$, $H^\ast{\rm Hom}(\zbb S^\dcal_{\rm sub}(X), A)$, and $H^\ast {\rm Hom}(\zbb S^\dcal_{\rm aff}(X), A)$ have natural commutative graded $A$-algebra structures and the isomorphisms between them are isomorphisms of graded $A$-algebras.
	\end{itemize}
\begin{proof}
	\begin{itemize}
		\item[{\rm (1)}] The result is immediate from Proposition \ref{natural}.
		\item[{\rm (2)}] See \cite[Remark 3.8(2)]{smh} for $H^\ast(X; A)$. The argument there also applies to $H^\ast {\rm Hom} (\zbb S^\dcal_{\rm sub}(X), A)$ and $H^\ast {\rm Hom}(\zbb S^\dcal_{\rm aff}(X), A)$. Since the cohomology isomorphisms in Part 1 are induced by the natural simplicial maps
$$S^\dcal_{\rm aff}(X) \longrightarrow S^\dcal_{\rm sub}(X) \longhookrightarrow S^\dcal(X) ,$$
they are isomorphisms of graded $A$-algebras.\qedhere
	\end{itemize}
\end{proof}
\end{cor}

\begin{rem}\label{cubic}
	In the study of differential forms and de Rham cohomology of diffeological spaces, Iglesias-Zemmour introduced the complex $\bvec{\mathsf{C}}_\star (X)$ of reduced groups of cubic chains for a diffeological space $X$, and called its homology $\bvec{\mathsf{H}}_{\ast}(X)$ the cubic homology of $X$ (\cite[pp. 182-183]{IZ}).
	
	We can easily see that $\bvec{\mathsf{H}}_\ast(X)$ is a smooth homotopy invariant. In fact, given a smooth homotopy $H:\rbb\times X\longrightarrow Y$ between $f$ and $g$, a chain homotopy $H_\sharp : \bvec{\mathsf{C}}_\ast(X)\longrightarrow \bvec{\mathsf{C}}_{\ast+1}(Y)$ between $\bvec{\mathsf{C}}_\ast(f)$ and $\bvec{\mathsf{C}}_\ast(g)$ is defined by
	\[
		\rbb^p \overset{\sigma}{\longrightarrow} X \longmapsto \rbb^{p+1} = \rbb \times \rbb^p \overset{1\times \sigma}{\longrightarrow} \rbb \times X \overset{H}{\longrightarrow} Y.
	\]
	Thus, by an argument similar to that in the proof of Proposition \ref{natural}, we can use \cite[Theorem II]{EM} to construct a natural chain homotopy equivalence between $\bvec{\mathsf{C}}_\ast(X)$ and $\zbb S^\dcal(X)$, showing that $\bvec{\mathsf{H}}_\ast(X)$ is naturally isomorphic to $H_\ast(X)$.
\end{rem}

The basic idea of the proof that $\zbb S^\dcal(X)$, $\zbb S^\dcal_{\rm sub}(X)$, and $\bvec{\mathsf{C}}_\ast(X)$ are chain homotopy equivalent was briefly discussed in \cite[Remark 3.9]{smh}. It is also shown in \cite[Section 4.1]{Kuri20} that $\zbb S^\dcal_{\rm aff}(X)$, $\zbb S^\dcal_{\rm sub}(X)$, and $\bvec{\mathsf{C}}_\ast(X)$ are chain homotopy equivalent.

\subsection{Diffeological coverings}
The notion of a diffeological fiber bundle is a generalization of that of a locally trivial fiber bundle, and is defined by local triviality of the pullback along any plot (see \cite[8.9]{IZ}). A diffeological fiber bundle with discrete fiber is called a {\sl diffeological covering}.

Similarly, a simplicial fiber bundle is defined by triviality of the pullback along any map from $\Delta[p]$ $(p \ge 0)$ (see \cite[Definition 11.8]{May}). A simplicial fiber bundle with discrete fiber is called a {\sl simplicial covering}.

We prove the following result, which is used in the proof of Theorem \ref{thm1.1}.

\begin{prop}\label{cover}
	The singular functors $S^\dcal$, $S^\dcal_{\rm sub}$, and $S^\dcal_{\rm aff}$ transform diffeological coverings with fiber $F$ to simplicial coverings with fiber $F$. Hence, a diffeological covering $\pi:E\longrightarrow X$ with fiber $F$ defines the natural morphisms of simplicial coverings with fiber $F$
	\begin{center}
		\begin{tikzcd}
		S^\dcal_{\rm aff}(E) \arrow[r] \arrow[d, "S^\dcal_{\rm aff}(\pi)" swap] & S^\dcal_{\rm sub}(E) \arrow[r, hook] \arrow[d, "S^\dcal_{\rm sub}(\pi)" swap] & S^\dcal(E) \arrow[d, "S^\dcal (\pi)" swap]\\
		S^\dcal_{\rm aff}(X) \arrow[r] & S^\dcal_{\rm sub}(X) \arrow[r, hook] & S^\dcal(X).
		\end{tikzcd}
	\end{center}
\begin{proof}
	We prove the result in three steps. \vspace{1ex}\\  
	Step 1. We show that $S^\dcal(\pi): S^\dcal(E) \longrightarrow S^\dcal(X)$ is a simplicial covering with fiber $F$.\\
	Assume given a map $k:\Delta[p] \longrightarrow S^\dcal(X)$ and let $\kappa: \Delta^p \longrightarrow X$ be the smooth map corresponding to $k$. Noticing that $\Delta^p$ is smoothly contractible (Lemma \ref{simplex}(1)), we then have a pullback diagram in $\dcal$
		\begin{center}
			\begin{tikzcd}
			\Delta^p \times F \arrow[r] \arrow[d, "proj" swap] & E \arrow[d, "\pi"]\\
			\Delta^p \arrow[r, "\kappa"] & X
			\end{tikzcd}
		\end{center}
		(see \cite[p. 264]{IZ}). Note that $S^\dcal$ is a right adjoint and consider the commutative diagram in $\scal$ consisting of two pullback squares
		\begin{center}
			\begin{tikzcd}
			\Delta[p]\times S^\dcal(F) \arrow[r] \arrow[d, "proj" swap] & S^\dcal(\Delta^p)\times S^\dcal(F) \arrow[r] \arrow[d, "proj" swap] & S^\dcal(E)\arrow[d, "S^\dcal(\pi)" swap]\\
			\Delta[p] \arrow[r] & S^\dcal(\Delta^p) \arrow[r, "S^\dcal(\kappa)"] & S^\dcal(X),
			\end{tikzcd}
		\end{center}
		where $\Delta[p]\longrightarrow S^\dcal(\Delta^p)$ is the map corresponding to the $p$-simplex $1_{\Delta^p}$ of $S^\dcal(\Delta^p)$. Then, the outer rectangle gives the desired local triviality of $S^\dcal(\pi)$ (see \cite[Exercise 8 on page 72]{Mac}). \vspace{1ex}\\
\hspace{-0.3em}Step 2. Note that $\Delta^p_{\rm sub}$ and $\abb^p$ are smoothly contractible (Lemma \ref{simplex}(1)) and that $S^\dcal_{\rm sub}$ and $S^\dcal_{\rm aff}$ are right adjoints (Remark \ref{extra}(1)). Then, by an argument similar to that in Step 1, we can see that $S^\dcal_{\rm sub}(\pi)$ and $S^\dcal_{\rm aff}(\pi)$ are also simplicial coverings with fiber $F$.\vspace{1ex}\\ 
\hspace{-0.27em}Step 3. The natural morphisms of simplicial coverings are defined by Proposition \ref{natural}.
\end{proof}
\end{prop}

\section{Weak equivalences between smooth singular complexes}\label{section3}
In this section, we prove Theorem \ref{thm1.1} and Corollary \ref{cor1.2}, using results of Section 3.

 The main statement of Theorem \ref{thm1.1} is divided into the following two parts:
  \begin{itemize}
   \item[(I)] The natural map $S^{\dcal}_{\mathrm{sub}}(X)\longhookrightarrow S^{\dcal}(X)$ is a weak equivalence.
   
   \item[(II)] The natural map $S^{\dcal}_{\mathrm{aff}}(X)\longrightarrow S^{\dcal}(X)$ is a weak equivalence.
  \end{itemize}
 After constructing a fibrant approximation functor for the category of simplicial sets in Section \ref{section3.1}, we prove  Parts I and II in Sections \ref{section3.2} and \ref{section3.3}, respectively. We complete the proofs of Theorem \ref{thm1.1} and Corollary \ref{cor1.2} in Section 4.4.
 
 %We also briefly discuss another smooth singular complex, denoted by $S^\dcal_{\rm ext}(X)$, and the relation to $S^\dcal(X)$, $S^\dcal_{\rm sub}(X)$, and $S^\dcal_{\rm aff}(X)$.

\subsection{Fibrant approximation to a simplicial set}\label{section3.1}
 The category $\scal$ of simplicial sets is a cofibrantly generated model category having
 \[
  \jcal_{\scal} = \{\Lambda_k[p]\longhookrightarrow \Delta[p] \mid p>0, 0\le k \le p \}
 \]
 as a set of generating trivial cofibrations. Applying the infinite glueing construction \cite[pp. 104-105]{DS} for $\jcal_S$ to a simplicial map $\varphi: K\to L$, we thus obtain the factorization
 \begin{center}
  \begin{tikzcd}[column sep=4em]
   K \arrow[r, "i"] \arrow[rd, "\varphi"'] & K' \arrow[d, "p"] \\
   & L
  \end{tikzcd}
 \end{center}
 where $i$ is a trivial cofibration and $p$ is a fibration. However, since every simplicial map to the terminal object $\ast$ has a right lifting property for $\Lambda_k [1] \longhookrightarrow \Delta[1]$ ($k = 0, 1$), we can construct a fibrant approximation $K \hspace{0.2em} \mathbf{\hat{}}$ of $K$ by applying the infinite glueing construction for
 \[
  \jcal'_{\scal} = \{\Lambda_k[p] \longhookrightarrow \Delta[p] \mid p>1, 0\le k\le p \}
 \]
 to $K\to \ast$. Let $\scal_f$ denote the full subcategory of $\scal$ consisting of fibrant objects (i.e., Kan complexes). Then, the functor $\cdot \ \mathbf{\hat{}}: \scal\to \scal_f$ is a fibrant approximation functor, for which $K \hspace{0.15em}\mathbf{\hat{}}_0 = K_0$ holds. An attachment of $\Delta[2]$ along $\Lambda_k[2]$ adds one nondegenerate $2$-simplex and one nondegenerate $1$-simplex, which correspond to the basic $2$-simplex of $\Delta[2]$ and its $k$-th face respectively.

\subsection{Proof of Part I}\label{section3.2}
 Let us begin by reducing the proof to simpler cases. First, consider the decomposition $X = \coprod X_\alpha$ into connected components (see \cite[pp. 105-107]{IZ}). Since
 \[
   S^{\dcal}_\text{sub}(X) = \coprod S^{\dcal}_{\text{sub}}(X_\alpha) \text{ and } S^\dcal(X) = \coprod S^\dcal(X_\alpha),
 \]
 we may assume that $X$ is connected.

 Next, consider the universal covering $\varpi: Z\to X$ (see \cite[p. 264]{IZ}). By Proposition \ref{cover}, we then have the morphism of simplicial coverings with fiber $\pi^{\dcal}_1(X)$
 \begin{center}
  \begin{tikzcd}[column sep=4em]
   \pi^{\dcal}_1(X) \arrow[r, "="'] \arrow[d] & \pi^{\dcal}_1(X) \arrow[d] \\
   S^{\dcal}_{\mathrm{sub}}(Z) \arrow[r, hook] \arrow[d] & S^{\dcal}(Z) \arrow[d] \\
   S^{\dcal}_{\mathrm{sub}}(X) \arrow[r, hook] & S^{\dcal}(X).
  \end{tikzcd}
 \end{center}
 Hence, we may assume that $X$ is $1$-connected (note that $S^{\dcal}_{\mathrm{sub}}(X)$ need not be a Kan complex and use \cite[Theorem 4.2 in Chapter III]{GZ}).

Since $S^\dcal(X)$ is Kan (Corollary \ref{lbryhorn}(2)), the inclusion $S^\dcal_{\rm sub}(X)\longhookrightarrow S^\dcal(X)$ extends to a map $S^\dcal_{\rm sub}(X) \hspace{0.2em} \mathbf{\hat{}} \longrightarrow S^\dcal(X)$, which induces an isomorphism on the homology (see Corollary \ref{quism}). Thus, we have only to show that
 \[
  \pi_1(S^{\dcal}_{\mathrm{sub}}(X) \hspace{0.2em} \mathbf{\hat{}} ,x_0) = 0
 \]
 for any fixed $x_0 \in X$ (see Theorem \ref{homotopygp} and \cite[Theorem 13.9]{May}).

 Recall from \cite[p. 8]{GJ} and \cite[Lemma 16.3]{May} the following facts on the topological realization functor $|\;|: \scal\to \czero$:
  \begin{itemize}
   \item The topological realization $|K|$ of a simplicial set $K$ is a $CW$-complex having one $n$-cell for each non-degenerate $n$-simplex of $K$.
   
   \item For a pointed Kan complex $(K, k_0)$, the simplicial fundamental group $\pi_1(K,k_0)$ is naturally isomorphic to the topological fundamental group $\pi_1(|K|,k_0)$.
  \end{itemize}
 For a simplicial set $K$, $NK_n$ denotes the set of non-degenerate $n$-simplices of $K$. The $n$-cell of $|K|$ corresponding to $\sigma\in NK_n$ is also denoted by $\sigma$. The $1$-cell $\sigma$ of $|K|$ is endowed with the canonical orientation; the $1$-cell $\sigma$ endowed with the reverse orientation is denoted by $\bar{\sigma}$. We also use the standard notation ${\rm sk}_n\: K$ for the $n$-skeleton of $K$.

 From these facts and the construction of the fibrant approximation $K \hspace{0.2em} \mathbf{\hat{}}$ of $K$, we see the following:
  \begin{itemize}
   \item $\pi_1(S^{\dcal}_{\rm sub}(X) \hspace{0.2em} \mathbf{\hat{}}, x_0) \cong \pi_1(|{\rm sk}_2\: S^\dcal_{\mathrm{sub}}(X) \hspace{0.2em} \mathbf{\hat{}} \hspace{0.15em}|, x_0)$.
   
	\item Every element of $\pi_1(|{\rm sk}_2\: S^{\dcal}_{\mathrm{sub}}(X) \hspace{0.2em} \mathbf{\hat{}} \hspace{0.15em}|, x_0) $ can be represented by a continuous map
	 \[
	  \omega: (\Delta^1_{\mathrm{top}}, \dot{\Delta}^1_{\mathrm{top}})\longrightarrow (|{\rm sk}_1\: S^{\dcal}_{\mathrm{sub}}(X)\hspace{0.2em} \mathbf{\hat{}} \hspace{0.15em}|, x_0).
	 \]
	 Further, $\omega$ can be chosen as the concatenation of finitely many $1$-cells $\tau_1, \ldots, \tau_l$, where $\tau_j = \sigma_j$ or $\bar{\sigma_j}$ for some $\sigma_j \in NS^\dcal_{\rm sub}(X)_1$.
  \end{itemize}
 We would like to simplify the expression $\tau_1 \cdots \tau_l$ for $\omega$ and show that $\omega$ is null homotopic rel $\dot{\Delta}^1_{\rm top}$.
 
  A smooth $1$-simplex $\sigma:\Delta^1_\text{sub}\longrightarrow X$ of a diffeological space $X$ is called {\it tame} if $\sigma$ is constant near each vertex. By the following lemma, we may assume that each $\sigma_j$ is tame.
  \begin{lem}\label{lem3.1}
  	Let $X$ be a diffeological space and $\sigma$ a $1$-simplex of $S^{\dcal}_{\mathrm{sub}}(X)$. Then, there exists a $2$-simplex $\Sigma$ of $S^{\dcal}_{{\rm sub}}(X)$ such that $d_0 \Sigma$ is the constant map to $\sigma((1))$, $d_1 \Sigma$ is tame, and $d_2 \Sigma = \sigma$.
  \end{lem}
  
  \begin{proof}
    We choose a non-decreasing smooth function $\mu: [0,1]\to [0,1]$ such that $\mu\equiv 0$ near $0$ and $\mu\equiv 1$ near $1$, and construct the desired $2$-simplex $\Sigma$ of $S^\dcal_{\rm sub}(X)$ in two steps.\vspace{-2ex}\\
    
  		\noindent Step 1: {\sl Construction of $F:\Delta^2_{\text{sub}}\to \Delta^2_{\text{sub}}$}. We construct a smooth map $F: \Delta^2_{\text{sub}}\to \Delta^2_\text{sub}$ (i.e., a $2$-simplex $F$ of $S^\dcal_{\rm sub}(\Delta^2_{\rm sub})$) satisfying the following condition: 
  		\begin{itemize}
  			\item[$\bullet$] Each $d_i F$ corestricts to the $i$-th face of $\Delta^2_{\text{sub}}$ and the corestriction of $d_i F$ is identified with
  			\[
  			\begin{cases*}
  			id & \text{for $i = 0,2,$} \\
  			\mu & \text{for $i = 1,$}
  			\end{cases*}
  			\]
  			in a canonical manner.
  		\end{itemize}
  		Set $U = \left\{ (x_0, x_1, x_2)\in \Delta^2_{\text{sub}} \mid 0\le x_1 < \frac{1}{2} \right\}$. Choose a non-increasing smooth function $\phi: \left[ 0, \frac{1}{2} \right]\to [0,1]$ such that $\phi\equiv 1$ near $0$ and $\phi\equiv 0$ near $\frac{1}{2}$. Under the identification
  		\begin{eqnarray*}
  			U & \xrightarrow[\ \ \cong \ \ ]{} & [0,1] \times \left[ 0, \frac{1}{2} \right), \\
  			(x_0, x_1, x_2) & \longmapsto & \left( \frac{x_2}{1-x_1}, x_1 \right)
  		\end{eqnarray*}
  		define the self-map $U \xrightarrow{\ \ f \ \ } U$ by
  		\[
  		f(x,y) = (\phi(y)\mu(x)+(1-\phi(y))x, y).
  		\]
  		Then, the desired map $\Delta^2_{\text{sub}}\xrightarrow{\ \ F \ \ } \Delta^2_\text{sub}$ is defined by
  		\[
  		F = 
  		\begin{cases}
  		f \ \ \text{ on } U\\
  		id \ \ \text{ outside } U.
  		\end{cases}
  		\]
  		Step 2: {\sl Construction of $\Sigma: \Delta^2_{\text{sub}}\to X$}. The desired $2$-simplex $\Sigma$ of $S^{\dcal}_{\text{sub}}(X)$ is defined to be the composite
  		\[
  		\Delta^2_{\text{sub}} \xrightarrow{\ \ F \ \ } \Delta^2_{\text{sub}} \xrightarrow{\ \ s^1 \ \ } \Delta^1_\text{sub} \xrightarrow{\ \ \sigma \ \ } X,
  		\]
  	where $s^1$ is defined by $s^1(x_0, x_1, x_2) = (x_0, x_1+x_2)$.
  \end{proof}
  Second, let us see that $\omega$ can be chosen as the concatenation of $\sigma_1, \ldots, \sigma_l$ for some $\sigma_1,\ldots,\sigma_l \in NS^\dcal_{\rm sub}(X)_1$. For this, consider $\Sigma_j \in S^\dcal_{\rm sub}(X)_2$ defined to be the composite
 \[
  \Delta^2_{\rm sub} \xrightarrow{\ \ s \ \ } \Delta^1_{\rm sub} \xrightarrow{\ \ \sigma_j \ \ } X,
 \]
 where $s(x_0, x_1, x_2) = (x_0+x_2, x_1)$. Then, $d_2 \Sigma_j = \sigma_j$, $d_1 \Sigma_j$ is constant, and ${\sigma_j}' := d_0 \Sigma_j$ satisfies ${\sigma_j}'(t)=\sigma_j(1-t)$. Thus, if $\tau_j = \bar{\sigma_j}$, then we can replace $\tau_j$ with $\sigma'_j$. Hence, we may assume that $\omega$ is the concatenation of $\sigma_1, \ldots, \sigma_l$.

 Third, let us see that $\omega$ can be chosen as the continuous map corresponding to a single tame $1$-simplex $\sigma$ of $S^{\dcal}_{\mathrm{sub}}(X)$. For this, we first consider the extension problem in $\dcal$
 \begin{center}
  \begin{tikzcd}[column sep=4em]
   \Lambda^2_{1\; \mathrm{sub}} \arrow[r, "\sigma_2 + \sigma_1"] \arrow[d, hook'] & X \\
   \Delta^2_{\mathrm{sub}}, \arrow[ru, dashed]
  \end{tikzcd}
 \end{center} 
where $\sigma_2 + \sigma_1: \Lambda^2_{1\: \mathrm{sub}} \longrightarrow X$ is defined to be $\sigma_2$ on the 0th face and $\sigma_1$ on the 2nd face. (the smoothness of $\sigma_2+\sigma_1$ follows from the tameness of $\sigma_1$ and $\sigma_2$). We define the map $\Sigma: \Delta^2_{\mathrm{sub}}\to X$ to be the composite
 \[
  \Delta^2_{\mathrm{sub}} \xrightarrow{\ \ r \ \ } \Lambda^2_{1\; \mathrm{sub}} \xrightarrow{\ \sigma_2 + \sigma_1 \ } X,
 \]
 where $r$ is the continuous retraction onto $\Lambda^2_{1\; \mathrm{sub}}$ depicted in Fig. \ref{fig3.1}. Noticing that $\sigma_1$ and $\sigma_2$ are tame, we can easily see that $\Sigma$ is a solution of the extension problem in $\dcal$ such that $\eta:= d_1\Sigma$ is also tame. Thus, $\omega$ can be chosen as the concatenation of $\eta,\sigma_3, \ \ldots,\sigma_l$. By iterating this procedure, we may assume that $\omega$ is the continuous map corresponding to a single tame $1$-simplex $\sigma$ of $S^{\dcal}_{\mathrm{sub}}(X)$.\vspace{-2.0ex}\\
 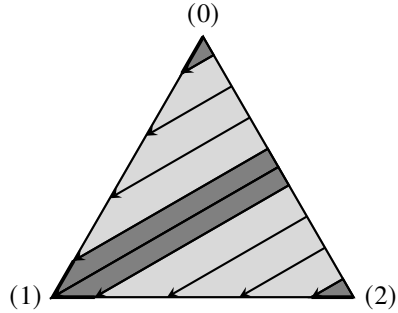
\begin{figure}[htbp]
  \centering
  \begin{tikzpicture}[thick, baseline=0pt]
   \coordinate[label=left:$(1)$]  (1) at (0,0);
   \coordinate[label=right:$(2)$] (2) at (4,0);
   \coordinate[label=above:$(0)$] (0) at (2,3.464);
		
   \coordinate[](a) at ($ (0)!.07!(2)$);
   \coordinate[](b) at ($ (0)!.19!(2)$);
   \coordinate[](c) at ($ (0)!.31!(2)$);
   \coordinate[](d) at ($ (0)!.43!(2)$);
   \coordinate[](e) at ($ (0)!.5!(2)$);
   \coordinate[](f) at ($ (0)!.57!(2)$);
   \coordinate[](g) at ($ (0)!.69!(2)$);
   \coordinate[](h) at ($ (0)!.81!(2)$);
   \coordinate[](i) at ($ (0)!.93!(2)$);
		
   \coordinate[](j) at ($ (0)!.14!(1)$);
   \coordinate[](k) at ($ (0)!.38!(1)$);
   \coordinate[](l) at ($ (0)!.62!(1)$);
   \coordinate[](m) at ($ (0)!.86!(1)$);
		
   \coordinate[](n) at ($(1)!.14!(2)$);
   \coordinate[](o) at ($(1)!.38!(2)$);
   \coordinate[](p) at ($(1)!.62!(2)$);
   \coordinate[](q) at ($(1)!.86!(2)$);
		
   \filldraw[fill=gray] (m)--(1)--(e)--(d)--cycle;
   \filldraw[fill=gray] (1)--(n)--(f)--(e)--cycle;
   \filldraw[fill=gray] (0)--(j)--(a)--cycle;
   \filldraw[fill=gray] (2)--(q)--(i)--cycle;
		
   \filldraw[fill=gray, opacity=0.3] (j)--(m)--(d)--(a)--cycle;
   \filldraw[fill=gray, opacity=0.3] (f)--(i)--(q)--(n)--cycle;
		
   \draw [thick, line width=0.5mm] (m) -- (1) -- (n) node []{};
   \draw [thick, line width=0.5mm] (0) -- (j) node []{};
   \draw [thick, line width=0.5mm] (q) -- (2) node []{};

   \draw [thick, -stealth] (a) -- (j) node []{};
   \draw [thick, -stealth] (b) -- (k) node []{};
   \draw [thick, -stealth] (c) -- (l) node []{};
   \draw [thick, -stealth] (d) -- (m) node []{};
   \draw [thick, -stealth] (e) -- (1) node []{};
		
   \draw [thick, -stealth] (f) -- (n) node []{};
   \draw [thick, -stealth] (g) -- (o) node []{};
   \draw [thick, -stealth] (h) -- (p) node []{};
   \draw [thick, -stealth] (i) -- (q) node []{};
		
   \path[clip, draw] (1) -- (2) -- (0) -- cycle;
  \end{tikzpicture}
  \caption{The retraction $r$}
  \label{fig3.1}
 \end{figure}

\noindent

 Last, let us see that $\omega$ is null homotopic rel $\dot{\Delta}^1_{\rm top}$. Since $X$ is $1$-connected, the extension problem in $\dcal$
 \begin{center}
  \begin{tikzcd}[column sep=4em]
   \dot{\Delta}^2 \arrow[r, "\sigma + 0 + 0"] \arrow[d, hook'] & X \\
   \Delta^2 \arrow[ru, dashed]
  \end{tikzcd}
 \end{center}
\hspace{1em}\vspace{-3ex}

\noindent has a solution $\Sigma$, where $0$ denotes the constant map to the base point $x_0$ (see Theorem \ref{homotopygp}).
 
 Now, we recall the smooth map $\psi^2_0: \Delta^2 \longrightarrow \Delta^2$ from \cite[Steps 1-3 in the proof of Theorem 8.6]{smh}. For $0<\epsilon < \frac{1}{2}$, set $V_i(\epsilon)=\{(x_0,x_1,x_2)\in \Delta^2\:|\: x_i > 1-\epsilon\}$. For a given $\epsilon_0$ with $0<\epsilon_0 < \frac{1}{2}$, the smooth map
 \[
 	\psi^2_0: \Delta^2 \longrightarrow \Delta^2
 \]
 is constructed so that it satisfies the following conditions:
 \begin{itemize}
 	\item $\psi^2_0$ preserves each closed simplex of $\Delta^2$.
 	\item $\psi^2_0$ maps each $V_i(\frac{\epsilon_0}{2})$ to the vertex $(i)$.
 	\item $\psi^2_0$ coincides with $1_{\Delta^2}$ on $\Delta^2 \backslash \cup V_i(\epsilon_0)$.
 \end{itemize}
Thus, we see from Lemma \ref{simplex}(3) that $\psi^2_0: \Delta^2_{\rm sub}\longrightarrow \Delta^2$ is smooth.

Consider the smooth map $\psi^2_0:\Delta^2_{\rm sub}\longrightarrow \Delta^2$ defined for sufficiently small $\epsilon_0 > 0$, and define the $2$-simplex $\Sigma'$ of $S^\dcal_{\mathrm{sub}}(X)$ to be the composite
 \[
  \Delta^2_{\mathrm{sub}}\xrightarrow{\ \ \psi^2_0 \ \ } \Delta^2 \xrightarrow{\ \ \Sigma\ \ } X.
 \]
 Since $\Sigma' |_{\dot{\Delta}^2_{\mathrm{sub}}} = \sigma+0+0$, $\Sigma'$ yields a homotopy (rel $\dot{\Delta}^1_{\mathrm{top}}$) between $\omega$ and $0$, which completes the proof.

\subsection{Proof of Part II.}\label{section3.3}
By Proposition \ref{cover} and an argument similar to that in Section \ref{section3.2}, we may assume that $X$ is $1$-connected.
 
 Since $S^\dcal(X)$ is Kan (Corollary \ref{lbryhorn}(2)), the canonical map $S^\dcal_{\rm aff}(X)\longrightarrow S^\dcal(X)$ extends to a map $S^\dcal_{\rm aff}(X)\hspace{0.1em}\hat{} \longrightarrow S^\dcal(X)$, which induces an isomorphism on the homology (see Corollary \ref{quism}). Thus, we have only to show that
 \[
  \pi_1(S^{\dcal}_{\mathrm{aff}}(X) \hspace{0.15em} \mathbf{\hat{}}\hspace{0.1em},x_0) = 0
 \]
 for any fixed $x_0 \in X$ (see Theorem \ref{homotopygp} and \cite[Theorem 13.9]{May}).
 
 Similarly to the proof of Part I, we have the following:
  \begin{itemize}
   \item $\pi_1(S^{\dcal}_{\mathrm{aff}}(X)\hspace{0.15em} \mathbf{\hat{}}\hspace{0.1em}, x_0) \cong \pi_1(|{\rm sk}_2\: S^\dcal_{\mathrm{aff}}(X)\hspace{0.15em} \mathbf{\hat{}}\hspace{0.15em}|, x_0)$.
   
	\item Every element of $\pi_1(|{\rm sk}_2 \ S^{\dcal}_{\mathrm{aff}}(X)\hspace{0.15em} \mathbf{\hat{}}\hspace{0.15em}|, x_0) $ can be represented by a continuous map
	 \[
	  \omega: (\Delta^1_{\mathrm{top}}, \dot{\Delta}^1_{\mathrm{top}})\longrightarrow (|{\rm sk}_1\: S^{\dcal}_{\mathrm{aff}}(X)\hspace{0.15em} \mathbf{\hat{}}\hspace{0.15em}|, x_0).
	 \]
	 Further, $\omega$ can be chosen as the concatenation of finitely many $1$-cells $\tau_1, \cdots, \tau_l$, where $\tau_j = \sigma_j$ or $\bar{\sigma_j}$ for some $\sigma_j \in NS^\dcal_{\mathrm{aff}}(X)_1$.
  \end{itemize} 
We would like to simplify the expression $\tau_1 \cdots \tau_l$ for $\omega$ and show that $\omega$ is null homotopic rel $\dot{\Delta}^1_{\rm top}$.

A smooth $1$-simplex $\sigma: \abb^1\to X$ of a diffeological space $X$ is 
called {\it tame} if $\sigma$ is constant near $(-\infty, 0]$ and near $[1, \infty)$, where $\abb^1$ is identified with $\rbb$ in a canonical manner. By the following analogue of Lemma \ref{lem3.1}, we may assume that each $\sigma_j$ is tame.
\begin{lem}\label{lem3.2}
	Let $X$ be a diffeological space and $\sigma$ a $1$-simplex of $S^{\dcal}_{\mathrm{aff}}(X)$. Then, there exists a $2$-simplex $\Sigma$ of $S^{\dcal}_{\mathrm{aff}}(X)$ such that $d_0 \Sigma$ is the constant map to $\sigma((1))$, $d_1 \Sigma$ is tame, and $d_2 \Sigma = \sigma$.  
\begin{proof}
	Set $U = \left\{ (x_0, x_1, x_2)\in \abb^2 \mid -\frac{1}{2} < x_1 < \frac{1}{2} \right\}$. Choose a non-decreasing smooth function $\mu: \rbb\to [0, 1]$ such that $\mu \equiv 0$ near $(-\infty, 0]$ and $\mu \equiv 1$ near $[1, \infty)$, and a smooth function $\phi: [-\frac{1}{2}, \frac{1}{2}] \to [0, 1]$ such that $\phi \equiv 1$ near $0$ and $\phi \equiv 0$ near $\{-\frac{1}{2}, \frac{1}{2}\}$. Then, we can construct the desired $2$-simplex $\Sigma$ in a manner similar to that in the proof of Lemma \ref{lem3.1}.
\end{proof}
\end{lem}
Second, let us see that $\omega$ can be chosen as the concatenation of $\sigma_1, \ldots, \sigma_l$ for some $\sigma_1,\ldots,\sigma_l \in NS^\dcal_{\rm aff}(X)_1$. For this, consider $\Sigma_j \in S^\dcal_{\rm aff}(X)_2$ defined to be the composite
\[
	\abb^2 \xrightarrow{\ \ s\ \ } \abb^1 \xrightarrow{\ \ \sigma_j\ \ } X,
\]
where $s(x_0,x_1,x_2)=(x_0+x_2,x_1)$. Then, $d_2\Sigma_j=\sigma_j, d_1\Sigma_j$ is constant, and $\sigma'_j := d_0\Sigma_j$ satisfies $\sigma'_j(t)=\sigma_j(1-t)$. Thus, if $\tau_j = \bar{\sigma_j}$, then we can replace $\tau_j$ with $\sigma'_j$. Hence, we may assume that $\omega$ is the concatenation of $\sigma_1,\ldots, \sigma_l$.

Next, we show the following lemma. For $i=0, 1, 2, d^i: \abb^1 \longrightarrow \abb^2$ denotes the obvious affine extension of $d^i: \Delta^1 \longrightarrow \Delta^2$ (see Section 2.2).
 
\begin{lem}\label{lem3.3}
 Let $X$ be a diffeological space and $\gamma_0$, $\gamma_1$, and $\gamma_2$ tame $1$-simplices of $S^{\dcal}_{\mathrm{aff}}(X)$ such that $d_0 \gamma_2 = d_1 \gamma_0$, $d_0\gamma_0 = d_0 \gamma_1$, and $d_1\gamma_1 = d_1 \gamma_2$. If the extension problem in $\dcal$
 \begin{center}
  \begin{tikzcd}[column sep=4em]
   \dot{\Delta}^2 \arrow[r, "\sum \gamma_i|_{\Delta^1}"] \arrow[d, hook'] & X \\
   \Delta^2 \arrow[ru, dashed] & 
  \end{tikzcd}
 \end{center}
 has a solution, then the extension problem in $\dcal$ 
 \begin{center}
  \begin{tikzcd}[column sep=4em]
   \cup\: d^i \abb^1 \arrow[r, "\sum \gamma_i"] \arrow[d, hook'] & X \\
   \abb^2 \arrow[ru, dashed] & 
  \end{tikzcd}
 \end{center}
 also has a solution.
\end{lem}

\begin{proof}
	We choose a solution $\Sigma$ of the first extension problem, and use the smooth map $\psi^2:\Delta^2 \longrightarrow \Delta^2$ constructed in \cite[Steps 1-3 in the proof of Theorem 8.6]{smh} to modify and extend $\Sigma$.
	
	To describe the basic properties of $\psi^2$, we use the following notations. For a continuous self-map $f$ of $\Delta^p$, we set
	\[
		\mathrm{carr}_{k}\ f  =  \{ x \in \Delta^p \ | \ f(x) \neq x,\ f(x) \in \mathrm{sk}_{k}\ \Delta^{p} \}	\hspace{0.3cm} \text{and} \hspace{0.3cm}  \mathrm{supp}_{k}\ f  =  \overline{\mathrm{carr}_{k}\ f}.
	\]
	Further, for a subset $\{i_0,\ldots,i_k\}$ of $\{0,\ldots,p\}$, we set
	\begin{gather*}
		V_{\{i_0,\ldots, i_k\}}=\{(x_0,\ldots,x_p)\in \Delta^p\:|\: x_i > x_j \text{ for }  i\in \{i_0,\ldots,i_k\}\text{ and } j \notin \{i_0,\ldots,i_k\}\},\\
		({\rm supp}_k\: f)^\circ_{\{i_0,\ldots, i_k\}} = ({\rm supp}_k\:f)^\circ \cap V_{\{i_0,\ldots,i_k\}}.
	\end{gather*}
	
	For a given $\epsilon_0$ with $0<\epsilon_0 < \frac{1}{2}$, the smooth maps $\psi^2_k:\Delta^2 \longrightarrow \Delta^2$ $(k=0,1)$ are defined so that they satisfy the following conditions:
	\begin{itemize}
		\item $\psi^2_k$ preserves each closed simplex of $\Delta^2$ and $\psi^2_k = id$ on ${\rm sk}_k\: \Delta^2$ (note that ${\rm sk}_0\: \Delta^2= \{(0),(1),(2)\}$ and that ${\rm sk}_1\: \Delta^2 = \dot{\Delta}^2$).
		\item $({\rm supp}_0\: \psi^2_0)^\circ_{\{i\}} = V_i(\frac{\epsilon_0}{2})$ and $\psi^2_0 = id$ on $\Delta^2 \backslash \cup V_i(\epsilon_0)$ (see Section 4.2).
		\item $({\rm supp}_0\: \psi^2_0)^\circ \cup ({\rm supp}_1\: \psi^2_1)^\circ$ is a neighborhood of $\dot{\Delta}^2$.
		\item $({\rm supp}_1\: \psi^2_1)^\circ=({\rm supp}_1\: \psi^2_1)^\circ_{\{1,2\}}\: \coprod \ ({\rm supp}_1\:\psi^2_1)^\circ_{\{0,2\}} \coprod \ ({\rm supp}_1\: \psi^2_1)^\circ_{\{0,1\}}$, and $\psi^2_1$ preserves each $V_i(\frac{\epsilon_0}{2})$ and maps a point $x$ of $({\rm supp}_1\: \psi^2_1)^\circ_{\{i_0,i_1\}}$ to the intersection of the $i$-th face of $\Delta^2$ and the line through the vertex $(i)$ and $x$, where $i\neq i_0,i_1$.
	\end{itemize}
(See Fig. 2). The map $\psi^2:\Delta^2\longrightarrow \Delta^2$ is defined to be the composite
\[
	\Delta^2 \xrightarrow{\ \ \psi^2_1\ \ } \Delta^2 \xrightarrow{\ \ \psi^2_0\ \ } \Delta^2.
\]

Consider the smooth map $\psi^2:\Delta^2 \longrightarrow \Delta^2$ for a sufficiently small $\epsilon_0 > 0$ and define $\Sigma'$ to be the composite
 \[
 	  \Delta^2\xrightarrow{\ \ \psi^2 \  \ } \Delta^2 \xrightarrow{\ \ \Sigma \ \ } X.
 \]
 Then, $\Sigma'$ has the following properties:
 \begin{itemize}
 	\item $\Sigma'|_{\dot{\Delta}^2} = \Sigma|_{\dot{\Delta}^2}$.
 	\item $\Sigma'|_{{({\rm supp}_0\: \psi^2_0)^\circ_{\{i\}}}}$ is constant.
 	\item $\Sigma'|_{({\rm supp}_1\:\psi^2_1)^\circ_{\{i_0,i_1\}}}$ is constant along any ray from the vertex $(i)$ with $i\neq i_0,i_1$.
 \end{itemize}

\begin{figure}[H]
	\centering
  \begin{tikzpicture}[thick, baseline=0pt]
	\coordinate[label=above left:$({\rm supp}_0\: \psi^2_0)^\circ_{\{1\}}$]  (1) at (0,0);
	\coordinate[label=above right:$({\rm supp}_0\: \psi^2_0)^\circ_{\{2\}}$] (2) at (4,0);
	\coordinate[label=right:$({\rm supp}_0\: \psi^2_0)^\circ_{\{0\}}$] (0) at (2,3.464);
	
	\coordinate[label=above:$(0)$] (A0) at (2.0,3.8);
	\coordinate[label=below left:$(1)$] (A1) at (-0.2,0);
	\coordinate[label=below right:$(2)$] (A2) at (4.2,0);
	
	\coordinate[label=below:$({\rm supp}_1\: \psi^2_1)^\circ_{\{1,2\}}$] (B0) at (2.3,-0.3);
	\coordinate[label=right:$({\rm supp}_1\: \psi^2_1)^\circ_{\{0,2\}}$]  (B1) at (3.3,1.732);
	\coordinate[label=left:$({\rm supp}_1\: \psi^2_1)^\circ_{\{0,1\}}$]  (B2) at (0.8,1.732);
	
	\coordinate[](11) at (-1,0);
	\coordinate[](12) at (-0.5,-0.866);
	\coordinate[](01) at (1.5,4.33); 
	\coordinate[](02) at (2.5,4.33);
	\coordinate[](21) at (5,0);
	\coordinate[](22) at (4.5,-0.866);
	
	\coordinate[](a) at ($ (0)!.24!(2)$);
	\coordinate[](b) at ($ (0)!.38!(2)$);
	\coordinate[](c) at ($ (0)!.62!(2)$);
	\coordinate[](d) at ($ (0)!.76!(2)$);
	
	\coordinate[](j) at ($ (0)!.24!(1)$);
	\coordinate[](k) at ($ (0)!.38!(1)$);
	\coordinate[](l) at ($ (0)!.62!(1)$);
	\coordinate[](m) at ($ (0)!.76!(1)$);
	
	\coordinate[](n) at ($(1)!.24!(2)$);
	\coordinate[](o) at ($(1)!.38!(2)$);
	\coordinate[](p) at ($(1)!.62!(2)$);
	\coordinate[](q) at ($(1)!.76!(2)$);
	
	\coordinate[](r) at ($(1)!.14!(2)$);
	\coordinate[](s) at ($(1)!.86!(2)$);
	\coordinate[](t) at (3.33,0.18);
	\coordinate[](u) at (0.67,0.18);
	
	\coordinate[](v) at ($(0)!.14!(2)$);
	\coordinate[](w) at ($(0)!.86!(2)$);
	\coordinate[](x) at (3.5,0.48);
	\coordinate[](y) at (2.17,2.795);
	
	\coordinate[](z) at ($(0)!.14!(1)$);
	\coordinate[](aa) at ($(0)!.86!(1)$);
	\coordinate[](ab) at (0.5,0.48);
	\coordinate[](ac) at (1.83,2.795);
	
	\filldraw[fill=gray, opacity=0.5] (m) -- (1) -- (n)--cycle;
	\filldraw[fill=gray, opacity=0.5] (a) -- (0) -- (j)--cycle;
	\filldraw[fill=gray, opacity=0.5] (q) -- (2) -- (d)--cycle;
	
	\filldraw[fill=gray, opacity=0.5] (r) -- (s) -- (t) -- (u) --cycle;
	\filldraw[fill=gray, opacity=0.5] (v) -- (w) -- (x) -- (y) --cycle;
	\filldraw[fill=gray, opacity=0.5] (z) -- (aa) -- (ab) -- (ac) --cycle;
	
	\filldraw[fill=gray, opacity=0] (1) -- (2) -- (0)--cycle;
	
	\draw [thick] (m) -- (1) -- (n) node []{};
	\draw [thick] (11) -- (1) -- (12) node []{};
	
	\draw [thick] (a) -- (0) -- (j) node []{};
	\draw [thick] (01) -- (0) -- (02) node []{};
	
	\draw [thick] (q) -- (2) -- (d) node []{};
	\draw [thick] (21) -- (2) -- (22) node []{};
	
	\path[clip, draw] (1) -- (2) -- (0) -- cycle;
	\end{tikzpicture}\\
	Fig. 2
\end{figure}

 We thus extend $\Sigma'$ to $\abb^2$ as follows. Define $\Sigma'|_{A_i}$ to be constant for $i = 0$, $1$, $2$ (see Fig. 3). Define $\Sigma'|_{B_i}$ to be constant along any ray from the vertex $(i)$ (see Fig. 3). Then, we can easily see that $\Sigma': \abb^2\to X$ is the desired solution of the second extension problem. 
 \begin{figure}[H]
 	\centering
  \begin{tikzpicture}[thick, baseline=0pt]
   \coordinate[label=above left:$(1)$]  (1) at (0,0);
   \coordinate[label=above right:$(2)$] (2) at (4,0);
   \coordinate[label=left:$(0)$] (0) at (2,3.464);

   \coordinate[label=above:$A_0$] (A0) at (2.0,3.8);
   \coordinate[label=below left:$A_1$] (A1) at (-0.2,0);
   \coordinate[label=below right:$A_2$] (A2) at (4.2,0);

   \coordinate[label=left:$B_0$] (B0) at (2.3,-0.6);
   \coordinate[label=right:$B_1$]  (B1) at (4,1.732);
   \coordinate[label=left:$B_2$]  (B2) at (0,1.732);
   
   \coordinate[label=below: $\text{Fig. 3}$] (fig) at (2.0, -1);

   \coordinate[](11) at (-1,0);
   \coordinate[](12) at (-0.5,-0.866);
   \coordinate[](01) at (1.5,4.33); 
   \coordinate[](02) at (2.5,4.33);
   \coordinate[](21) at (5,0);
   \coordinate[](22) at (4.5,-0.866);
  
   \coordinate[](a) at ($ (0)!.24!(2)$);
   \coordinate[](b) at ($ (0)!.38!(2)$);
   \coordinate[](c) at ($ (0)!.62!(2)$);
   \coordinate[](d) at ($ (0)!.76!(2)$);
		
   \coordinate[](j) at ($ (0)!.24!(1)$);
   \coordinate[](k) at ($ (0)!.38!(1)$);
   \coordinate[](l) at ($ (0)!.62!(1)$);
   \coordinate[](m) at ($ (0)!.76!(1)$);
		
   \coordinate[](n) at ($(1)!.24!(2)$);
   \coordinate[](o) at ($(1)!.38!(2)$);
   \coordinate[](p) at ($(1)!.62!(2)$);
   \coordinate[](q) at ($(1)!.76!(2)$);
		
   \filldraw[fill=gray, opacity=0.3] (1) -- (2) -- (0)--cycle;

   \filldraw[fill=gray] (m) -- (1) -- (n)--cycle;
   \filldraw[fill=gray] (a) -- (0) -- (j)--cycle;
   \filldraw[fill=gray] (q) -- (2) -- (d)--cycle;

   \draw [thick, line width=0.5mm] (m) -- (1) -- (n) node []{};
   \draw [thick, line width=0.5mm] (11) -- (1) -- (12) node []{};
   
   \draw [thick, line width=0.5mm] (a) -- (0) -- (j) node []{};
   \draw [thick, line width=0.5mm] (01) -- (0) -- (02) node []{};

   \draw [thick, line width=0.5mm] (q) -- (2) -- (d) node []{};
   \draw [thick, line width=0.5mm] (21) -- (2) -- (22) node []{};
		
   \path[clip, draw] (1) -- (2) -- (0) -- cycle;
  \end{tikzpicture}
 \end{figure}
\end{proof}
 
 Let us see that $\omega$ can be chosen as the continuous map corresponding to a single tame $1$-simplex $\sigma$ of $S^\dcal_{\rm aff}(X)$. For this, we first consider the extension problem in $\dcal$
 \begin{center}
	\begin{tikzcd}[column sep=4em]
	\Lambda^2_1 \arrow[r, "\sigma_2|_{\Delta^1}+\sigma_1|_{\Delta^1}"] \arrow[d, hook'] & X\\
	\Delta^2 \arrow[ru, dashed]
	\end{tikzcd}
\end{center}
Then, we can use the continuous retraction $r:\Delta^2 \longrightarrow \Lambda^2_1$ depicted in Fig. 1 to construct a solution $\Sigma: \Delta^2 \longrightarrow X$ such that the composite $\Delta^1 \xrightarrow{\ \ d^1\ \ } \Delta^2 \xrightarrow{\ \ \Sigma\ \ } X$ is constant near each vertex. Define the tame $1$-simplex $\eta$ of $S^\dcal_{\rm aff}(X)$ by $\eta|_{\Delta^1}=\Sigma\circ d^1$ and consider the extension problem in $\dcal$
 \begin{center}
	\begin{tikzcd}[column sep=4em]
	\cup\: d^i \abb^1 \arrow[r,"\sigma_2 + \eta + \sigma_1"] \arrow[d, hook']& X\\
	\abb^2 \arrow[ru, dashed]
	\end{tikzcd}
\end{center}
Since this extension problem has a solution (see Lemma \ref{lem3.3}), $\omega$ can be chosen as the concatenation of $\eta,\sigma_3,\ldots,\sigma_l$. By iterating this procedure, we may assume that $\omega$ is the continuous map corresponding to a single tame $1$-simplex $\sigma$ of $S^\dcal_{\rm aff}(X)$.

Last, let us see that $\omega$ is null homotopic rel $\dot{\Delta}^1_{\rm top}$. Since $X$ is $1$-connected, the extension problem in $\dcal$
 \begin{center}
	\begin{tikzcd}[column sep=4em]
	\dot{\Delta}^2 \arrow[r, "\sigma|_{\Delta^1} + 0 + 0"] \arrow[d, hook'] & X\\
	\Delta^2 \arrow[ru,dashed]
	\end{tikzcd}
\end{center}
has a solution (see Theorem \ref{homotopygp}). Hence, the extension problem in $\dcal$
 \begin{center}
	\begin{tikzcd}[column sep=4em]
	\cup\: d^i \abb^1 \arrow[r, "\sigma+0+0"] \arrow[d, hook'] & X\\
	\abb^2 \arrow[ru, dashed]
	\end{tikzcd}
\end{center}
also has a solution (Lemma \ref{lem3.3}), which shows that $\omega$ is null homotopic rel $\dot{\Delta}^1_{\rm top}$.

\subsection{Proofs of Theorem \ref{thm1.1} and Corollary \ref{cor1.2}}
In this subsection, we complete the proofs of Theorem \ref{thm1.1} and Corollary \ref{cor1.2}.

\begin{proof}[Proof of Theorem \ref{thm1.1}]
	The proof of the main statement is given in Sections \ref{section3.2}-\ref{section3.3}. Since $S^{\dcal}(X)$ is always fibrant (Corollary \ref{lbryhorn}(2)), the last statement is obvious.
\end{proof}

Let $\scal_{\ast}$ denote the category of pointed simplicial sets, and let $\scal_{\ast f}$ denote the full subcategory of $\scal_{\ast}$ consisting of fibrant objects (i.e., pointed Kan complexes). Choosing a fibrant approximation functor $R: \scal_{\ast}\longrightarrow \scal_{\ast f}$, we define the $i$-th homotopy group functor $\pi_i: \scal_{\ast}\to Gr$ to be the composite 
\[
\scal_{\ast}\xrightarrow{\ \ R \ \ } \scal_{\ast f}\xrightarrow{\ \ \pi_i\ \ } Gr.
\]
(Strictly speaking, $\pi_0$ is defined as a $Set_\ast$-valued functor, where $Set_\ast$ denotes the category of pointed sets.) Then, up to natural isomorphisms, the functor $\pi_i:\scal_{\ast}\longrightarrow Gr$ extends the original homotopy group functor $\pi_i: \scal_{\ast f} \longrightarrow Gr$ and the extension $\pi_i:\scal_{\ast} \longrightarrow Gr$ is independent of the choice of $R$. Further, we can see that if a fibrant approximation $K \longrightarrow K'$ and a point $k$ of $K$ are given, then $\pi_i(K,k)$ is canonically isomorphic to the $i$-th homotopy group of the pointed Kan complex $(K',k)$.
\begin{proof}[Proof of Corollary \ref{cor1.2}]
	The result follows immediately from Theorems \ref{thm1.1} and \ref{homotopygp}.
\end{proof}
\section{Diffeological principal bundles}\label{section3.4}
In this section, we recall the notions of a diffeological principal bundle and a simplicial principal bundle (Section 5.1) and establish Theorem \ref{bdletriviality}, which characterizes diffeological principal bundles using the singular functor $S^\dcal_{\rm aff}$ (Section 5.2).

\subsection{Diffeological and simplicial principal bundles}
In this subsection, we recall the three notions of principal bundles in $\dcal$; the weakest notion is due to Iglesias-Zemmour (see Definition \ref{diffbdle}(2)). We also make a brief review on simplicial principal bundles.

Let $\ccal$ be a category with finite products, and $G$ a group in $\ccal$. Then, $\ccal G$ denotes the category of right $G$-objects of $\ccal$. For $B\in \ccal$, $\ccal G/B$ denotes the category of objects of $\ccal G$ over $B$, where $B$ is regarded as an object of $\ccal G$ with trivial $G$-action.
\begin{defn}\label{diffbdle}
	Let $G$ be a diffeological group, and $X$ a diffeological space.
	\begin{itemize}
		\item[{\rm (1)}] An object $\pi:E\longrightarrow X$ of $\dcal G/X$ is called a {\it locally trivial principal $G$-bundle} if there exists an open cover $\{U_i\}$ of $X$ such that for each $i$, a pullback diagram in $\dcal$
		\begin{center}
			\begin{tikzcd}
			U_i\times G \arrow[r, hook] \arrow[d, "proj" swap] & E \arrow[d, "\pi"]\\
			U_i \arrow[r, hook] & X
			\end{tikzcd}
		\end{center}
	with equivariant upper arrow exists; such an open cover $\{U_i\}$ is called a {\it trivialization open cover} of $\pi: E\longrightarrow X$.
	\hspace{1.0em}An object $\pi:E\longrightarrow X$ of $\dcal G/X$ is called a {\it $\dcal$-numerable principal $G$-bundle} if $\pi$ admits a $\dcal$-numerable trivialization open cover (i.e., a trivialization open cover $\{U_i\}$ which admits a smooth partition of unity subordinate to it).
	\item[{\rm (2)}] An object $\pi:E\longrightarrow X$ of $\dcal G/X$ is called a {\it diffeological principal $G$-bundle} if for any plot $p:U\longrightarrow X$, the pullback $p^\ast E\longrightarrow U$ is a locally trivial principal $G$-bundle.
	\item[{\rm (3)}] A morphism between locally trivial (or diffeological) principal $G$-bundles $\pi:E\longrightarrow X$ and $\pi':E'\longrightarrow X'$ is a commutative diagram in $\dcal G$ of the form
	\begin{eqnarray}
		\begin{tikzcd}
		E \arrow[r, "\hat{f}"] \arrow[d, "\pi" swap] & E' \arrow[d, "\pi'"]\\
		X \arrow[r, "f"] & X'.
		\end{tikzcd}
	\end{eqnarray}
	Note that (5.1) is necessarily a pullback diagram in $\dcal$ (see \cite[8.13 Note 2]{IZ}). The categories of locally trivial principal $G$-bundles, $\dcal$-numerable principal $G$-bundles, and diffeological principal $G$-bundles are denoted by $\mathsf{P}\dcal G$, $\mathsf{P}\dcal G_{\rm num}$, and $\mathsf{P}\dcal G_{\rm diff}$, respectively.
	\end{itemize}
\end{defn}
 We have the obvious fully faithful embeddings
 \[
 	\mathsf{P}\dcal G_{\rm num} \longhookrightarrow \mathsf{P}\dcal G \longhookrightarrow \mathsf{P}\dcal G_{\rm diff}.
 \]
 We see from the following examples that the two inclusions are proper. Recall from \cite[8.15]{IZ} that for a diffeological group $G$ and its diffeological subgroup $H$, the quotient map $\pi: G \longrightarrow G/H$ is a diffeological principal $H$-bundle.
\begin{exa}\label{proper}
 \begin{itemize}
  \item[{\rm (1)}] Let $\gamma: \zbb^m \longrightarrow \rbb^n$ be a monomorphism of abelian groups with $\Gamma:= \gamma(\zbb^m)$ dense. Then, the quotient diffeological group $T_\Gamma = \rbb^n/\Gamma$ is called an \textit{irrational torus}.
Since the underling topology of $T_{\Gamma}$ is indiscrete, the diffeological principal $\zbb^{m}$-bundle $\pi: \rbb^n \longrightarrow T_{\Gamma}$ is not locally trivial. 
  \item[{\rm (2)}] Christensen and Wu constructed a nontrivial locally trivial principal $\rbb^{>0}$-bundle $\pi: P \longrightarrow X$ with $X \simeq_{\dcal} \ast$ (see \cite[Example 3.12]{CW17}). By \cite[Theorem 5.10]{CW17}, the locally trivial principal $\rbb^{>0}$-bundle $\pi$ is not $\dcal$-numerable.
 \end{itemize}
\end{exa}

To study diffeological principal bundles, we also need the  notion of a simplicial principal bundle \cite[Chapter IV]{May}.

\begin{defn}\label{simplicialbdle}
	Let $H$ be a simplicial group, and $K$ a simplicial set.
\begin{itemize}
%	\item[{\rm (1)}] A simplicial map $\pi:E\longrightarrow K$ is called an $F$-bundle if for any map $k:\Delta[p]\longrightarrow K$, there exists a pullback diagram in $\scal$
%	\begin{center}
%		\begin{tikzcd}
%		\Delta[p]\times F \arrow[r, "\kappa"] \arrow[d, "proj" swap] & E \arrow[d]\\
%		\Delta[p] \arrow[r, "k"] & K.
%		\end{tikzcd}
%	\end{center}
	\item[{\rm (1)}] An object $\pi:E\longrightarrow K$ of $\scal H/K$ is called a {\it principal $H$-bundle} if for any map $k:\Delta[p]\longrightarrow K$, there exists a pullback diagram
	\begin{eqnarray*}
\begin{tikzcd}[column sep=4em]
\Delta[p]\times H \arrow[r, "\hat{k}"] \arrow[d, "proj" swap] & E \arrow[d]\\
\Delta [p] \arrow[r, "k"] & K
\end{tikzcd}
\end{eqnarray*}
	with $\hat{k}$ equivariant.
	\item[{\rm (2)}] A morphism between principal $H$-bundles $\pi:E\longrightarrow K$ and $\pi':E'\longrightarrow K'$ is a commutative diagram in $\scal H$ of the form
	\begin{eqnarray}
		\begin{tikzcd}
			E \arrow[r, "\hat{f}"] \arrow[d, "\pi" swap] & E' \arrow[d, "\pi'"]\\
			K \arrow[r, "f"] & K'.
		\end{tikzcd}
	\end{eqnarray}
	Note that (5.2) is necessarily a pullback diagram in $\scal$. The category of principal $H$-bundles are denoted by $\mathsf{P}\scal H$.
\end{itemize}
\end{defn}

\begin{rem}\label{principal}
	An object $\pi:E\longrightarrow K$ of $\scal H/K$ is a principal $H$-bundle if and only if the action of $H$ on $E$ is free and $\pi$ induces the isomorphism $E/H \overset{\bar{\pi}}{\underset{\cong}{\longrightarrow}}K$ (cf. \cite[Definition 18.1]{May}).
\end{rem}

 Let $\cdot_0 : \scal \longrightarrow Set$ denote the 0th component functor, which is naturally isomorphic to the functor $\scal(\Delta[0], \cdot)$. The following simple result is used in the proof of Theorem \ref{bdletriviality}.
\begin{lem}\label{zero}
	\begin{itemize}
		\item[{\rm (1)}] The composite
		\[
			\dcal \overset{S^\dcal_{\rm aff}}{\longrightarrow} \scal \overset{\cdot_0}{\longrightarrow} Set
		\]
		is naturally isomorphic to the underlying set functor for $\dcal$.
		\item[{\rm (2)}] The functor $\cdot_0: \scal \longrightarrow Set$ is a right adjoint.
	\end{itemize}
	\begin{proof}
		\begin{itemize}
			\item[{\rm (1)}] Obvious.
			\item[{\rm (2)}] Define the functor $d: Set \longrightarrow \scal$ to assign to a set $A$ the discrete simplicial set whose 0th component is $A$. Then, we can easily see that $(d, \cdot_0)$ is an adjoint pair.
		\end{itemize}
	\end{proof}
\end{lem}
For a given set $A$, the discrete simplicial set $d A$ is usually denoted by $A$.

\subsection{Proof of Theorem \ref{bdletriviality}}
In this subsection, we prove Theorem \ref{bdletriviality}; we begin by proving the ``only if" part of Part 1, and Part 2, and then prove the ``if" part of Part 1.

Recall that $S^\dcal_{\rm aff}$ is a right adjoint (Remark \ref{extra}(1)). Then, we see that $S^\dcal_{\rm aff}(G)$ is a simplicial group and that $S^\dcal_{\rm aff}(\pi):S^\dcal_{\rm aff}(P)\longrightarrow S^\dcal_{\rm aff}(X)$ is an object of $\scal S^\dcal_{\rm aff}(G)/ S^\dcal_{\rm aff}(X)$.

	\begin{proof}[Proof of the ``only if" part of Theorem \ref{bdletriviality}(1)]
		Assume given a map $k:\Delta[p] \longrightarrow S^\dcal_{\rm aff}(X)$ and let $\kappa: \abb^p \longrightarrow X$ be the smooth map corresponding to $k$. Then, we have a pullback diagram in $\dcal$
		\begin{center}
			\begin{tikzcd}
				\abb^p \times G \arrow[r] \arrow[d, "proj" swap] & P \arrow[d, "\pi"]\\
				\abb^p \arrow[r, "\kappa"] & X
			\end{tikzcd}
		\end{center}
		with equivariant upper arrow (see \cite[8.19]{IZ}). Note that $S^\dcal_{\rm aff}$ is a right adjoint and consider the commutative diagram in $\scal$ consisting of two pullback squares with equivariant upper arrows
		\begin{center}
			\begin{tikzcd}
				\Delta[p]\times S^\dcal_{\rm aff}(G) \arrow[r] \arrow[d, "proj" swap] & S^\dcal_{\rm aff}(\abb^p)\times S^\dcal_{\rm aff}(G) \arrow[r] \arrow[d, "proj" swap] & S^\dcal_{\rm aff}(P)\arrow[d, "S^\dcal_{\rm aff}(\pi)" swap]\\
				\Delta[p] \arrow[r] & S^\dcal_{\rm aff}(\abb^p) \arrow[r, "S^\dcal_{\rm aff}(\kappa)"] & S^\dcal_{\rm aff}(X),
			\end{tikzcd}
		\end{center}
		where $\Delta[p]\longrightarrow S^\dcal_{\rm aff}(\abb^p)$ is the map corresponding to the $p$-simplex $1_{\abb^p}$ of $S^\dcal_{\rm aff}(\abb^p)$. Then, the outer rectangle gives the desired local triviality of $S^\dcal_{\rm aff}(\pi)$ (see \cite[Exercise 8 on page 72]{Mac}).
	\end{proof}
	\begin{proof}[Proof of Theorem \ref{bdletriviality}(2).]
	Noting that $S^\dcal_{\rm aff}$ is a right adjoint, we see from Part 1 that $S^{\dcal}_{\rm aff}$ induces a functor from $\mathsf{P} \dcal G_{\rm diff}$ to $\mathsf{P} \scal S^\dcal_{\rm aff}(G)$. The faithfulness of the functor follows from Lemma \ref{zero}(1).
	\end{proof}

\begin{rem}\label{notfull}
	The functor $S^\dcal_{\rm aff}: \mathsf{P} \dcal G_{\rm diff} \longrightarrow \mathsf{P} \scal S^\dcal_{\rm aff}(G)$ need not be fully faithful. In fact, let $\pi:P \longrightarrow X$ be the locally trivial principal $\rbb^{>0}$-bundle in Example \ref{proper}(2), and let $\pi':P' \longrightarrow X$ be the trivial principal $\rbb^{>0}$-bundle. Since $X \simeq_{\dcal} \ast$, the diagram in $\dcal$ \vspace{-1.5ex}

\[
\begin{tikzpicture}
\node at (0,0)
{\begin{tikzcd}[column sep=huge]
	X \arrow[r] & \ast \arrow[hookrightarrow]{r}{x} &  X
	\end{tikzcd}};
\draw[->] (-2.2,-0.4)--(-2.2,-0.7)--(2.2,-0.7)--(2.2,-0.4);
\node [right] at (0,-1) {1x};
\end{tikzpicture}
\]\vspace{-4ex}

\noindent is commutative up to homotopy for  $x \in X$. Thus, by Lemma \ref{homotopypreserving}, the diagram in $\scal$\vspace{-1.5ex}

\[
\begin{tikzpicture}
\node at (0,0)
{\begin{tikzcd}[column sep=large]
	S^{\dcal}_{\rm aff}(X) \arrow[r] & \ast \arrow[hookrightarrow]{r}{x} & S^{\dcal}_{\rm aff}(X)
	\end{tikzcd}};
\draw[->] (-2.2,-0.4)--(-2.2,-0.7)--(2.2,-0.7)--(2.2,-0.4);
\node [right] at (0,-1) {$1_{S^{\dcal}_{\rm aff}(X)}$};
\end{tikzpicture}
\]\vspace{-4ex}

\noindent is also commutative up to homotopy. Hence, both $S^\dcal_{\rm aff}(P)$ and $S^\dcal_{\rm aff}(P')$ are trivial principal $S^\dcal_{\rm aff}(\rbb^{>0})$-bundles (see \cite[Corollary 20.6]{May}), which shows that $S^\dcal_{\rm aff}: \mathsf{P} \dcal \rbb^{>0} \longrightarrow \mathsf{P} \scal S^\dcal_{\rm aff}(\rbb^{>0})$ is not fully faithful. (From this argument, we also see that the faithful functor $S^{\dcal}_{\rm aff}: \dcal \longrightarrow \scal$ is not fully faithful.)
\end{rem}

Next, we prove the following lemma, which is used in the proof of ``if" part of Theorem \ref{bdletriviality}(1).

\begin{lem}\label{translation}
	Let $\pi: P \longrightarrow X$ be an object of $\dcal G/X$. Then, $\pi:P\longrightarrow X$ is a diffeological principal $G$-bundle if and only if $\pi$ satisfies the following conditions:
	\begin{itemize}
		\item[{\rm (i)}] $G$ acts on $P$ freely and $\pi:P\longrightarrow X$ induces a bijection $P/G\longrightarrow X$.
		\item[{\rm (ii)}] Given a solid arrow diagram in $\dcal$ \vspace{-2ex}
		\begin{center}
			\begin{tikzcd}
			& P \arrow[d, "\pi"]\\
			\abb^p \arrow[ur, dashed] \arrow[r, "\kappa"] & X,
			\end{tikzcd}
		\end{center}\vspace{-1ex}
		there exists a dotted arrow, making the diagram commute.
		\item[{\rm (iii)}] The translation function $\tau: P \underset{X}{\times} P \longrightarrow G$, defined by $u\cdot \tau(u,v) = v$, is smooth.
	\end{itemize}
	\begin{proof}
	  ($\Rightarrow$) {\sl Condition} (i). Obvious.\vspace{-2ex}\\

          	\noindent {\sl Condition} (ii). By \cite[8.19]{IZ},
		\begin{eqnarray}
		\kappa^\ast P \cong \abb^p \times G \text{ in } \dcal G /\abb^p
		\end{eqnarray}
		holds. Hence, $\pi$ satisfies condition (ii).\vspace{-2ex}\\
                
		\noindent {\sl Condition} (iii). We have only to show that $\tau: P \underset{X}{\times} P \longrightarrow G$ preserves global plots.
		
		Assume given a global plot $f: \abb^p \longrightarrow P \underset{X}{\times} P$. Since the components $f_1, f_2$ of $f$ are global plots of $P$ with $\pi \circ f_1 = \pi \circ f_2$, we have only to show that the composite
		\[
		\kappa^\ast P \underset{\abb^p}{\times} \kappa^\ast P \longrightarrow P \underset{X}{\times} P \xrightarrow{\ \ \tau\ \ } G
		\]
		is smooth, where $\kappa:= \pi\circ f_1 = \pi \circ f_2$. By (5.3), we have the identifications
		\[
		\kappa^\ast P \underset{\abb^p}{\times} \kappa^\ast P \cong (\abb^p \times G) \underset{\abb^p}{\times} (\abb^p \times G) \cong \abb^p \times G \times G,
		\]
		under which the composite $\kappa^\ast P \underset{\abb^p}{\times} \kappa^\ast P \longrightarrow P \underset{X}{\times} P \overset{\tau}{\longrightarrow} G$ is just the smooth map
		\[
		\abb^p \times G \times G \longrightarrow G
		\]
		sending $(x, g, h)$ to $g^{-1} h$.\vspace{-2ex}\\
                
		\noindent ($\Leftarrow$) Assume given a smooth map $\kappa: \abb^p \longrightarrow X$. By condition (ii), we can choose a section $\sigma$ of  the pullback $\kappa^\ast P \overset{\hat{\pi}}{\longrightarrow} \abb^p$ of $P\overset{\pi}{\longrightarrow} X$ along $\kappa$. Define the maps
		\[
		\abb^p \times G \overset{\phi_\kappa}{\underset{\psi_\kappa}{\rightleftarrows}} \kappa^\ast P
		\]
		by $\phi_\kappa(x, g) = \sigma(x)\cdot g$ and $\psi_\kappa(u)=(\hat{\pi}(u), \tau (\sigma(\hat{\pi}(u)), u))$, respectively. Then, we see that $\phi_\kappa$ and $\psi_\kappa$ are mutually inverses in $\dcal G/\abb^p$.
	\end{proof}
\end{lem}
We give a proof of the ``if" part of Theorem \ref{bdletriviality}(1), completing the proof of Theorem \ref{bdletriviality}.
\begin{proof}[Proof of the ``if" part of Theorem \ref{bdletriviality}(1)]
We have only to show that $\pi:P\longrightarrow X$ satisfies conditions (i)-(iii) in Lemma \ref{translation}. Throughout this proof, bear the following in mind: For a diffeological space $Z$,
\begin{itemize}
	\item $S^\dcal_{\rm aff}(Z)_0$ is just the set $Z$.
	\item $S^\dcal_{\rm aff}(Z)$ can be regarded as the set of global plots of $Z$.
\end{itemize}
Recall also that $S^\dcal_{\rm aff}$ is a right adjoint (see Remark \ref{extra}(1)).\vspace{-2.0ex}\\

\noindent {\sl Condition} (i). Consider the pullback diagram in $\dcal$
\begin{center}
	\begin{tikzcd}
	P_x \arrow[r] \arrow[d] & P \arrow[d, "\pi"]\\
	\{x\} \arrow[r] & X
	\end{tikzcd}
\end{center}
for $x\in X$. By applying the singular functor $S^\dcal_{\rm aff}$, we have the pullback diagram in $\scal$
\begin{center}
	\begin{tikzcd}
	S^\dcal_{\rm aff}(P_x) \arrow[r] \arrow[d] & S^\dcal_{\rm aff}(P) \arrow[d, "{S^\dcal_{\rm aff}} (\pi)"]\\
	\Delta[0] \arrow[r] & S^\dcal_{\rm aff}(X).
	\end{tikzcd}
\end{center}
Since $S^\dcal_{\rm aff}(\pi)$ is a principal $S^\dcal_{\rm aff}(G)$-bundle, $S^\dcal_{\rm aff}(P_x) \cong S^\dcal_{\rm aff}(G)$ in $\scal S^\dcal_{\rm aff}(G)$, and hence $P_x \cong G$ in $Set\: G$ holds (see Lemma \ref{zero}), which shows that $\pi$ satisfies condition (i).\vspace{-2.0ex}\\

%{\sl Condition (ii)}. For the principal $S^\dcal_{\rm aff}(G)$-bundle $S^\dcal_{\rm aff}(\pi): S^\dcal_{\rm aff}(P) \longrightarrow S^\dcal_{\rm aff}(X)$, the simplicial translation function
%\[
%	S^\dcal_{\rm aff}(P) \underset{S^\dcal_{\rm aff}(X)}{\times} S^\dcal_{\rm aff}(P) \overset{t}{\longrightarrow} S^\dcal_{\rm aff}(G)
%\]
%is defined in the obvious manner, Since the action of $S^\dcal_{\rm aff}(G)$ on $S^\dcal_{\rm aff}(P)$ is induced by that of $G$ on $P$, the map $t$ is identified with $S^\dcal_{\rm aff}(\tau)$ under the isomorphism
%\[
%	S^\dcal_{\rm aff}(P \underset{X}{\times} P) \cong S^\dcal_{\rm aff}(P)\underset{S^\dcal_{\rm aff}(X)}{\times} S^\dcal_{\rm aff}(P).
%\]
%Thus, $\tau$ preserves plots, and hence is a smooth map.\\
\noindent {\sl Condition} (ii). Consider the pullback diagram in $\dcal$
\begin{center}
	\begin{tikzcd}
	\kappa^\ast P \arrow[r] \arrow[d] & P \arrow[d, "\pi"]\\
	\abb^p \arrow[r, "\kappa"] & X
	\end{tikzcd}
\end{center}
and let $k$ denote the simplicial map $\Delta[p] \longrightarrow S^\dcal_{\rm aff}(X)$ corresponding to $\kappa$. Then, we have the commutative diagram in $\scal$ consisting of two pullback squares
\begin{center}
\begin{tikzcd}
k^\ast S^\dcal_{\rm aff}(P) \arrow[r] \arrow[d] & S^\dcal_{\rm aff}(\kappa^\ast P) \arrow[d] \arrow[r] & S^\dcal_{\rm aff}(P)\arrow[d, "S^\dcal_{\rm aff}(\pi)"]\\
\Delta[p] \arrow[r] & S^\dcal_{\rm aff}(\abb^p) \arrow[r, "S^\dcal_{\rm aff}(\kappa)"] & S^\dcal_{\rm aff}(X),
\end{tikzcd}
\end{center}
where $\Delta[p]\longrightarrow S^\dcal_{\rm aff}(\abb^p)$ is the simplicial map corresponding to the $p$-simplex $1_{\abb^p}$ of $S^\dcal_{\rm aff}(\abb^p)$ (see \cite[Exercise 8 on page 72]{Mac}). Since $S^\dcal_{\rm aff}(\pi)$ is a simplicial $S^\dcal_{\rm aff}(G)$-bundle, $k^\ast S^\dcal_{\rm aff}(P)\longrightarrow \Delta[p]$ has a section $s$. Then, the composite
\[
\Delta[p] \xrightarrow{\ \ s\ \ } k^\ast S^\dcal_{\rm aff}(P)\longrightarrow S^\dcal_{\rm aff}(P)
\]
defines the desired lifting of $\kappa$ along $\pi$.\vspace{-2.0ex}\\

\noindent {\sl Condition} (iii). We show that the map $\tau: P\underset{X}{\times} P \longrightarrow G$ preserves global plots. Assume given a global plot $f=(f_1, f_2): \abb^p \longrightarrow P \underset{X}{\times} P$. Since $f_1$ and $f_2$ are global plots of $P$ with $\pi\circ f_1 = \pi \circ f_2$, we set $\kappa = \pi\circ f_1 = \pi \circ f_2$ and let $\sigma_1$ and $\sigma_2$ denote the sections of $\kappa^\ast P \longrightarrow \abb^p$ corresponding to $f_1$ and $f_2$, respectively. Then, $\sigma_1$ and $\sigma_2$ correspond to sections of $k^\ast S^\dcal_{\rm aff}(P)\longrightarrow \Delta[p]$, which are denoted by $s_1$ and $s_2$, respectively (see the verification of condition (ii)). Since the principal $S^\dcal_{\rm aff}(G)$-bundle $k^\ast S^\dcal_{\rm aff}(P)\longrightarrow \Delta[p]$ is trivial, there exists a unique $p$-simplex $g$ of $S^\dcal_{\rm aff}(G)$ such that $s_1 \cdot g = s_2$. We thus see that the composite
\[
\abb^p \xrightarrow{\ \ f\ \ } P \underset{X}{\times} P \xrightarrow{\ \ \tau\ \ } G
\]
is just the global plot $g$.
\end{proof}

\begin{rem}\label{bdletriviality2}
%	Similarly to a diffeological principal bundle (resp. simplicial principal bundle), a diffeological fiber bundle (resp. simplicial fiber bundle) is defined by local triviality of the pullback along any plot (resp. triviality of the pullback along any map from $\Delta[p]$ $(p \ge 0)$); see \cite[8.9]{IZ} and \cite[Definition 11.8]{May} for the precise definitions.
\begin{itemize}
	\item[{\rm (1)}]
	Recall the notion of a diffeological fiber bundle and that of a simplicial fiber bundle from Section 3.3. We can then use the argument in the proof of the ``only if" part of Theorem \ref{bdletriviality}(1) to prove the following: If $\pi:E\longrightarrow X$ is a diffeological fiber bundle, then $S^\dcal_{\rm aff}(\pi): S^\dcal_{\rm aff}(E) \longrightarrow S^\dcal_{\rm aff}(X)$ is a simplicial fiber bundle.

\vspace{-1.0ex}\hspace{1em}This result along with Theorem \ref{thm1.1} enables us to apply the Serre spectral sequence \cite[Section 32]{May} to diffeological fiber bundles (cf. \cite[Remark 3.8(3)]{smh}).
	\item[{\rm (2)}] If we restrict ourselves to locally trivial principal $G$-bundles (resp. locally trivial fiber bundles), then the ``only if" part of Theorem \ref{bdletriviality}(1) (resp. the result stated in Part 1) remains true for the functor $S^\dcal$ (instead of $S^\dcal_{\rm aff}$) (see \cite[Corollary 5.15(1)]{smh}).

	\item[{\rm (3)}] If we restrict ourselves to diffeological coverings, then the result stated in Part 1 remains true for the functor $S^{\dcal}$ (instead of $S^{\dcal}_{\rm aff}$); see Proposition \ref{cover}. Similarly, if $G$ is discrete, then Theorem \ref{bdletriviality} remains true for the functor $S^{\dcal}$ (instead of $S^{\dcal}_{\rm aff}$). 
\end{itemize}
\end{rem}

\section{Characteristic classes of diffeological principal bundles}
In this section, we first give a  criterion for a simplicial principal bundle to be universal (Section 6.1). We then use this criterion to determine the homotopy type of $S^\dcal(X)$ for a diffeological space $X$ which admits a diffeological principal bundle with contractible total space (Proposition \ref{classifyingsp}), applying it to the classifying space $BG$ of a diffeological group $G$ and pathological diffeological spaces such as irrational tori and $\rbb/\qbb$ (Section 6.2). We use the proof of Proposition \ref{classifyingsp} along with Theorems \ref{thm1.1} and \ref{bdletriviality} to prove Proposition \ref{cor1.3} (Section 6.3). We end this section by discussing the sets of characteristic classes for various classes of principal bundles and their relation (Section 6.4).

\subsection{Universal simplicial principal bundles}
In this subsection, we recall the basics of universal simplicial principal bundles and give a criterion for a simplicial principal bundle to be universal.

Let $H$ be a simplicial group. A principal $H$-bundle $\varpi:E\longrightarrow L$ is called {\it universal} if $L$ is Kan (i.e., fibrant in $\scal$) and the natural map
\begin{equation*}
\begin{alignedat}{4}
&[K,L] &\longrightarrow &\left\{
\begin{alignedat}{4}
&\text{isomorphism classes of}\\
&\text{principal $H$-bundles over $K$}
\end{alignedat}
\right\}, \\
&[f]&\longmapsto &\ \ \ [f^\ast E]
\end{alignedat}
\end{equation*}
is bijective; the base $L$ of a universal principal $H$-bundle $\varpi: E \longrightarrow L$ is called a {\sl classifying complex} of $H$. By a simple argument, a classifying complex of $H$ is unique up to homotopy.
 Recall that the $W$-construction $q:WH\longrightarrow \overline{W}H$ is a universal principal $H$-bundle (\cite[Section 21]{May} or \cite[Section 4 in Chapter V]{GJ}) and that $WH$ is contractible (\cite[Proposition 21.5]{May}).

\begin{lem}\label{universalpb}
	Let $H$ be a simplicial group, and $\varpi:E\longrightarrow L$ be a principal $H$-bundle. Then, the following are equivalent:
	\begin{itemize}
		\item[{\rm (i)}] $\varpi: E \longrightarrow L$ is universal.
		\item[{\rm (ii)}] $L$ is Kan and the canonical map $E\longrightarrow \ast$ is a weak equivalence.
		\item[{\rm (iii)}] $E$ is a contractible Kan complex.
	\end{itemize}
	
	\begin{proof}
		{\rm (ii)$\Leftrightarrow$(iii)} Noticing that $H$ is Kan \cite[Theorem 17.1]{May}, we see that $L$ is Kan if and only if $E$ is Kan (see \cite[Proposition 7.5]{May}), and hence that (ii)$\Leftrightarrow$(iii).
		\item[{\rm (i)$\Leftrightarrow$(iii)}] We have only to prove that under the assumption that $L$ is Kan,
		\[
		\varpi:E \longrightarrow L \text{ is universal } \Leftrightarrow E \text{ is contractible}
		\]
		(see \cite[Proposition 7.5]{May}).
		
		Since $q: WH\longrightarrow \overline{W}H$ is universal, we have a morphism of principal $H$-bundles
		\begin{eqnarray}
		\begin{tikzcd}[column sep=4em]
		E \arrow[r] \arrow[d, "\varpi" swap] & WH \arrow[d,"q"]\\
		L \arrow[r, "\varphi"] & \overline{W}H.
		\end{tikzcd}
		\end{eqnarray}
		Note that $H$ and the four simplicial sets in (6.1) are Kan and consider the morphism between the homotopy exact sequences induced by (6.1). Then, we have the equivalences
		\begin{align*}
		\varpi: E\longrightarrow L \text{ is universal} &\Leftrightarrow \varphi:L \longrightarrow \overline{W}H \text{ is a homotopy equivalence}\\ 
		&\Leftrightarrow E \text{ is contractible.}
		\qedhere
		\end{align*}
	\end{proof}
\end{lem}

\begin{rem}\label{reason}
Lemma \ref{universalpb} can be  regarded as a variant of \cite[Theorem 3.9 in Chapter V]{GJ}. However, we record this lemma along with its proof for the following two reasons: one reason is to avoid using the model structure on $\scal G$ (see \cite[Section V.2]{GJ}) and the other reason is to emphasize the importance of the fibrancy of the base (cf. the proof of Proposition \ref{classifyingsp}).
\end{rem}
\subsection{Diffeological principal bundles with contractible total space}
In this subsection, we determine the homotopy type of $S^\dcal(X)$ for a diffeological space $X$ which admits a diffeological principal bundle $\pi: E \longrightarrow X$ with $E$ weakly contractible. Here, a diffeological space $Z$ is called \textit{weakly contracible} if the canonical map $Z \longrightarrow \ast$ is a weak equivalence. We can easily see that the following equivalences hold:
\begin{eqnarray*}
Z \rm{\ is \ weakly\  contractible} &\Leftrightarrow& S^{\dcal}(Z) \simeq \ast\\
                           &\Leftrightarrow& \pi_{\ast}^{\dcal}(Z,z)=0 \ {\rm for \ any} \ z \in Z
\end{eqnarray*}
(see Remark \ref{modelstr}(1), Corollary \ref{lbryhorn}(2), and Theorem \ref{homotopygp}).
\begin{prop}\label{classifyingsp}
	Let $G$ be a diffeological group and $\pi: E\longrightarrow X$ a diffeological principal $G$-bundle with $E$ weakly contractible. Then, $S^\dcal(X)$ is a classifying complex of the simplicial group $S^\dcal(G)$.
	\begin{proof}
		By Theorem \ref{bdletriviality}, $S^\dcal_{\rm aff}(\pi): S^\dcal_{\rm aff}(E)\longrightarrow S^\dcal_{\rm aff}(X)$ is a principal $S^\dcal_{\rm aff}(G)$-bundle. Let us construct a principal $S^\dcal_{\rm aff}(G)$-bundle $S^\dcal_{\rm aff}(\pi)': S^\dcal_{\rm aff}(E)' \longrightarrow S^\dcal_{\rm aff}(X)\hspace{0.1em}\mathbf{\hat{}}$ (see Section 4.1) and a morphism of principal $S^\dcal_{\rm aff}(G)$-bundles
		\begin{eqnarray*}
		\begin{tikzcd}
		S^{\dcal}_{\mathrm{aff}}(E) \arrow[r, hook] \arrow[d, "S^{\dcal}_{\mathrm{aff}}(\pi)"'] & S^{\dcal}_{\mathrm{aff}}(E)'\arrow[d, "S^{\dcal}_{\mathrm{aff}}(\pi)'"] \\
		S^{\dcal}_{\mathrm{aff}}(X) \arrow[r, hook] & S^{\dcal}_{\mathrm{aff}}(X)\hspace{0.1em}\mathbf{\hat{}}.
		\end{tikzcd}
		\end{eqnarray*}
		First, choose a classifying map $\varphi_{E}:S^\dcal_{\rm aff}(X)\longrightarrow \overline{W}S^\dcal_{\rm aff}(G)$. Then, note that $\overline{W}S^\dcal_{\rm aff}(G)$ is Kan and choose an extension $\varphi'_{E}:S^\dcal_{\rm aff}(X)\hspace{0.1em}\mathbf{\hat{}} \longrightarrow \overline{W} S^\dcal_{\rm aff}(G)$. By setting $S^\dcal_{\rm aff} (E)' = {\varphi'_{E}}^\ast W S^\dcal_{\rm aff}(G)$, we then obtain the desired diagram. 
		
		Thus, we can use \cite[Theorem 4.2 in Chapter III]{GZ} to see that $S^\dcal_{\rm aff}(E)\longhookrightarrow S^\dcal_{\rm aff}(E)'$ is a weak equivalence. Noticing that $S^{\dcal}_{\rm aff}(E) \longrightarrow \ast$ is a weak equivalence (see Theorem \ref{thm1.1}), we see from Lemma \ref{universalpb} that $S^{\dcal}_{\mathrm{aff}}(\pi)': S^{\dcal}_{\mathrm{aff}}(E)'\to S^{\dcal}_{\mathrm{aff}}(X)\hspace{0.1em}\mathbf{\hat{}}$ \hspace{0.5em}is a universal principal $S^{\dcal}_{\mathrm{aff}}(G)$-bundle. Hence, $S^\dcal(X)$ is a classifying complex of $S^\dcal_{\rm aff}(G)$, and hence of $S^\dcal(G)$ (see Theorem \ref{thm1.1}).
	\end{proof}
\end{prop}

\begin{cor}\label{BG}
	Let $G$ be a diffeological group. Then, the singular complex $S^\dcal(BG)$ of the classifying space $BG$ is a classifying complex of the simplicial group $S^\dcal(G)$.
	\begin{proof}
		Recall from \cite[Corollary 5.5]{CW17} that $EG$ is smoothly contractible. Then, the result is immediate from Proposition \ref{classifyingsp}.
	\end{proof}
\end{cor}

\begin{cor}\label{aspherical}
		Suppose that $X$ is a pointed diffeological space which has the weakly contractible universal covering. Then, the singular complex $S^\dcal(X)$ is the Eilenberg-Mac Lane complex $K(\pi_1^\dcal(X), 1)$. In particular, the (co)homology of $X$ is just the (co)homology of the group $\pi_1^\dcal(X)$.
	\begin{proof}
		Recall from \cite[8.26]{IZ} that the universal covering $\pi: Z \longrightarrow X$ is a diffeological principal $\pi^{\dcal}_1(X)$-bundle. Then, the result follows from Proposition \ref{classifyingsp}.
	\end{proof}
\end{cor}

\begin{rem}\label{SD}
	\begin{itemize}
		\item[{\rm (1)}] We can prove Corollary \ref{BG}, using neither the functor $S^\dcal_{\rm aff}$ nor Theorem \ref{thm1.1}. In fact, by Remark \ref{bdletriviality2}(2) and Lemma \ref{universalpb}, $S^\dcal(\pi_G): S^\dcal(EG)\longrightarrow S^\dcal(BG)$ is a universal principal $S^\dcal(G)$-bundle. However, the construction in the proof of Proposition \ref{classifyingsp} is useful in the proof of Proposition \ref{cor1.3}.
		\item[{\rm (2)}] We can also prove Corollary \ref{aspherical}, using neither the functor $S^\dcal_{\rm aff}$ nor Theorem \ref{thm1.1}. In fact, Corollary \ref{aspherical} follows from Proposition \ref{cover}. Alternatively, Corollary \ref{aspherical} follows from \cite[8.24]{IZ} and Theorem \ref{homotopygp}.
	\end{itemize}
\end{rem}
Corollary \ref{aspherical} determines the homotopy type of $S^\dcal(X)$, and hence the (co)homology of $X$ for well-known homogeneous diffeological spaces $X$ such as irrational tori and $\rbb/\qbb$.

\begin{exa}\label{R/Q}
	\begin{itemize}
		\item[{\rm (1)}] Let $\gamma: \zbb^m \longrightarrow \rbb^n$ be a monomorphism of abelian groups with $\Gamma:= \gamma(\zbb^m)$ dense, and consider the irrational torus $T_{\Gamma}=\rbb^{n}/\Gamma$. By Corollary \ref{aspherical}, the singular complex $S^\dcal(T_\Gamma)$ of $T_\Gamma$ is just the $m$-dimensional torus $K(\zbb^m,1)$. Hence, $H^\ast(T_\Gamma; \zbb)\cong \Lambda(\zbb^m)$ holds.
		\item[{\rm (2)}] The singular complex $S^\dcal(\rbb/\qbb)$ of the quotient diffeological group $\rbb/\qbb$ is just the rationalized circle $K(\qbb,1)$, and hence $\widetilde{H}_\ast(\rbb/\qbb; \zbb) = H_1 (\rbb /\qbb ;\zbb )=\qbb$. More generally, let $A$ be a countable subgroup of $\fbb\:(= \rbb, \cbb)$. Then, the singular complex $S^\dcal(\fbb/ A)$ of the quotient diffeological group $\fbb/ A$ is just $K(A, 1)$.
	\end{itemize}
\end{exa}

\begin{rem}\label{Tgamma}
Iglesias-Zemmour and Kuribayashi obtained calculational results similar to Example \ref{R/Q}(1) for other cohomology theories of irrational tori (\cite[Corollary]{IZ20}, \cite[Proposition 3.2]{Kuri20'}, and \cite[Remark 2.9]{Kuri20}). On the other hand, the de Rham cohomology $H^{\ast}_{dR}(T_{\Gamma})$ is isomorphic to $\Lambda(\rbb^n)$ \cite[Exercise 119]{IZ}, which along with Example \ref{R/Q}(1), shows that the de Rham theorem does not hold for irrational tori. This motivates the study of a forthcoming paper \cite{dRt}.
\end{rem}

Next, we introduce new aspherical homogeneous diffeological spaces, using Corollary \ref{aspherical}.

\begin{exa}\label{unipotent}
	Let $k$ be a countable subfield of $\fbb\:(=\rbb, \cbb)$ (e.g. an algebraic number field or a countable extension of $\qbb$ such as $\overline{\qbb} \cap \rbb$ or $\overline{\qbb}$). For an algebraic group $G$ over $k$, we can consider the homogeneous diffeological space $G(\fbb)/G(k)$. 
	
	If $G$ is a unipotent algebraic group over $k$, then the exponential map ${\rm exp}: \mathfrak{g}\longrightarrow G$ is an isomorphism of algebraic varieties, where $\mathfrak{g}$ is the Lie algebra of $G$ (see \cite[p. 289]{Milne}). Thus, we have the diffeomorphism
	\[
	\mathfrak{g}(\fbb) \overset{\rm exp}{\underset{\cong}{\longrightarrow}} G(\fbb)
	\]
	and the universal covering
	\[
	G(\fbb) \longrightarrow G(\fbb)/ G(k)
	\]
	of $G(\fbb)/ G(k)$. Hence, by Corollary \ref{aspherical},
	\[
	S^\dcal(G(\fbb)/ G(k)))= K(G(k), 1)
	\]
	holds, so the (co)homology of $G(\fbb)/G(k)$ is that of the group $G(k)$. The group $U_n(k)$ of upper triangular unipotent matrices and the Hisenberg group $H_n(k)$ (see \cite[p. 54]{OVIII}) are typical examples of unipotent algebraic groups.

        Further if $G$ is defined over a subring $k_0$ of $k$, we have
        $$
        S^{\dcal}(G(\fbb)/G(k_0))=K(G(k_0),1).
        $$
        We are interested in the case where $k_0$ is the ring $\ocal_{k}$ of integers of an algebraic number field $k$. If $k$ is an algebraic number field of degree $n$ with $\qbb \subsetneqq k \subsetneqq \rbb$, then $k_0 \ (=\ocal_{k})$ is a finitely generated free $\zbb$-module of rank $n$, and hence is dense in $\rbb$.
\end{exa}

\subsection{Proof of Proposition \ref{cor1.3}}
In this subsection, we prove Proposition \ref{cor1.3}.

\begin{proof}[Proof of Proposition \ref{cor1.3}]
	Let $\pi_G:EG\longrightarrow BG$ denote the universal $\dcal$-numerable principal $G$-bundle constructed in \cite{CW17}. Then, by Theorem \ref{bdletriviality}(1), $S^{\dcal}_{\mathrm{aff}}(\pi_G): S^{\dcal}_{\mathrm{aff}}(EG)\to S^{\dcal}_{\mathrm{aff}}(BG)$ is a principal $S^\dcal_{\rm aff}(G)$-bundle.
	
	We prove the result in two steps.\vspace{2mm}\\
	Step 1. {\sl Construction of a universal principal $S^\dcal_{\rm aff}(G)$-bundle which is an extension of $S^\dcal_{\rm aff}(\pi_G)$}. Recall from \cite[Corollary 5.5]{CW17} that $EG$ is smoothly contractible. Then, by the proof of Proposition \ref{classifyingsp}, we have a universal principal $S^\dcal_{\rm aff}(G)$-bundle $S^\dcal_{\rm aff}(\pi_G)': S^\dcal_{\rm aff}(EG)'\longrightarrow S^\dcal_{\rm aff}(BG)\hspace{0.1em}\mathbf{\hat{}}$ and a morphism of a principal $S^\dcal_{\rm aff}(G)$-bundles
	\begin{eqnarray*}
	\begin{tikzcd}
		S^{\dcal}_{\mathrm{aff}}(EG) \arrow[r, hook] \arrow[d, "S^{\dcal}_{\mathrm{aff}}(\pi_G)"'] & S^{\dcal}_{\mathrm{aff}}(EG)'\arrow[d, "S^{\dcal}_{\mathrm{aff}}(\pi_G)'"] \\
		S^{\dcal}_{\mathrm{aff}}(BG) \arrow[r, hook] & S^{\dcal}_{\mathrm{aff}}(BG)\hspace{0.1em}\mathbf{\hat{}} .
	\end{tikzcd}
	\end{eqnarray*}
	Step 2. {\sl Definition of $\alpha(P)$}. Let $\pi:P\longrightarrow X$ be a diffeological principal $G$-bundle. Since $S^\dcal_{\rm aff}(\pi):S^\dcal_{\rm aff}(P)\longrightarrow S^\dcal_{\rm aff}(X)$ is a principal $S^\dcal_{\rm aff}(G)$-bundle (Theorem \ref{bdletriviality}(1)), we have a classifying map $\varphi_P:S^\dcal_{\rm aff}(X)\longrightarrow S^\dcal_{\rm aff}(BG)\hspace{0.1em}\mathbf{\hat{}}$.
	
	Note that $H^{\ast}(Z; A) := H^{\ast} {\rm Hom}(\zbb S^\dcal(Z), A)\cong H^{\ast}{\rm Hom}(\zbb S^\dcal_{\rm aff}(Z), A)$ (see Corollary \ref{quism}) and that  $H^{\ast}{\rm Hom}(\zbb K, A) \\ \cong H^{\ast}{\rm Hom}(\zbb K \hspace{0.15em} \mathbf{\hat{}}, A)$. Then, we can define $\alpha(P)\in H^k(X; A)$ by $\alpha(P) = \varphi_P^{\ast}\alpha$. We can use Theorem \ref{bdletriviality} to show that $\alpha(f^\ast P)=f^\ast \alpha(P)$ holds, and hence that $\alpha(\cdot)$ defines a characteristic class for diffeological principal $G$-bundles. 
	
	Similarly, we can use Theorem \ref{bdletriviality} to show that $\alpha(\cdot)$ extends the characteristic class $\alpha(\cdot)$ for $\dcal$-numerable principal $G$-bundles (see Section 1 for the definition).
\end{proof}

\begin{rem}\label{Kan}
	The author does not know whether $S^\dcal_{\rm aff}(BG)$ is always Kan. If $S^\dcal_{\rm aff}(BG)$ is always Kan, the proof of Proposition \ref{cor1.3} becomes simpler (see Lemma \ref{universalpb}).
\end{rem}

Let us apply Proposition \ref{cor1.3} to special cases.

\begin{exa}\label{special}
	\begin{itemize}
		\item[{\rm (1)}] Let $\pi: Z \longrightarrow X$ be a Galois covering with structure group $\Gamma$ (see \cite[p. 262]{IZ}). Then, for a given class $\alpha \in H^k(\Gamma; A)$ $(\cong H^k (B\Gamma; \:A))$, the class $\alpha(Z)\in H^k(X; A)$ is defined by Proposition \ref{cor1.3}.
		\item[{\rm (2)}] Let $G$ be a diffeological group and $H$ a diffeological subgroup of $G$. Then, for a given class $\alpha \in H^k(BH;\: A)$, the class $\alpha(G)\in H^k(G/H; A)$ is defined by Proposition \ref{cor1.3} (see \cite[8.15]{IZ}).
	\end{itemize}
\end{exa}

If a relevant diffeological principal bundle in Example \ref{special} happens to be $\dcal$-numerable, then the class at issue is just the image of $\alpha$ under the homomorphism induced by the classifying map. However, this is not the case in general. See the following example, which specializes both Parts 1 and 2 of Example \ref{special}.

\begin{exa}\label{irrational}
	Let $\gamma: \zbb^m \longrightarrow \rbb^n$ be a monomorphism of abelian groups with $\Gamma:= \gamma (\zbb^m)$ dense, and consider the diffeological principal $\zbb^m$-bundle $P := \rbb^n \xrightarrow{\ \pi \ } T_\Gamma$ over the irrational torus $T_\Gamma$ (see Examples \ref{R/Q}(1) and \ref{special}(2)); note that $T_\Gamma$ is a diffeological group  and that $\pi$ is the universal covering of $T_\Gamma$.
	
	Since $S^\dcal_{\rm aff}(T_\Gamma)$ is already Kan (see \cite[Proposition 4.30 or Theorem 4.34]{CW}), $S^\dcal_{\rm aff}(\pi): S^\dcal_{\rm aff}(P)\longrightarrow S^\dcal_{\rm aff}(T_\Gamma)$ is a universal principal $\zbb^m$-bundle (see Step 1 in the proof of Proposition \ref{cor1.3}), and hence, we have a classifying map $\varphi_P: S^\dcal_{\rm aff}(T_\Gamma)\longrightarrow S^\dcal_{\rm aff}(B \zbb^m)\hspace{0.1em}\mathbf{\hat{}}$ which is obviously a homotopy equivalence in $\scal$.
	
	Since $S^\dcal_{\rm aff}(B\zbb^m)\hspace{0.1em}\mathbf{\hat{}}$ is just the Eilenberg-MacLane complex $K(\zbb^m,1)$, $H^\ast(B\zbb^m ; A) \cong (\Lambda \zbb^m)\otimes A$ holds. Thus, for any $\alpha\in H^\ast(B\zbb^m; A)$, the characteristic class $\alpha(P)\in H^\ast(T_\Gamma; A)$ is just the image $\varphi^\ast_P(\alpha)$ under the isomorphism $H^\ast(T_\Gamma; A) \overset{\varphi^\ast_P}{\underset{\cong}{\longleftarrow}} H^\ast(B\zbb^m; A)$.

On the other hand, since $\pi: P \longrightarrow T_\Gamma$ is not locally trivial (see Example \ref{proper}(1)), $P$ has no classifying map to $B\zbb^m$. Further, every nonzero element $\beta \in \widetilde{H}^\ast (T_\Gamma; \: A)$ is not contained in the image of the homomorphism induced by any smooth map $f: T_\Gamma \longrightarrow B\zbb^m$. In fact, we have the commutative diagram
\begin{center}
	\begin{tikzcd}
		S^\dcal(T_\Gamma) \arrow[d] \arrow[r, "S^\dcal(f)"] & S^\dcal(B\zbb^m) \arrow[d]\\
		S(\widetilde{T}_\Gamma) \arrow[r, "S(\widetilde{f})"] & S(\widetilde{B\zbb^m})
	\end{tikzcd}
\end{center}
(see Section 2.3). Since $S^\dcal(B\zbb^m) \longrightarrow S(\widetilde{B\zbb^m})$ is a homotopy equivalence (see \cite[Corollary 5.16]{smh}) and $S(\widetilde{T}_\Gamma) \simeq \ast$, $S^\dcal(f)$ is homotopic to a constant map. (We actually show that $B\zbb^m$ is smoothly homotopy equivalent to the torus $T^m$, and hence that $f$ is smoothly homotopic to a constant map; see a forthcoming paper.)
\end{exa}

\subsection{Sets of characteristic classes for the classes $\mathsf{P}\dcal G$, $\mathsf{P}\dcal G_{\rm num}$, and $\mathsf{P}\dcal G_{\rm diff}$}
In this subsection, we discuss the sets of characteristic classes for the classes (or categories) $\mathsf{P}\dcal G$, $\mathsf{P}\dcal G_{\rm num}$, and $\mathsf{P}\dcal G_{\rm diff}$ (see Definition \ref{diffbdle}) and their relation.

Let $\pcal$ denote one of the categories $\mathsf{P}\dcal G$, $\mathsf{P}\dcal G_{\rm num}$, and $\mathsf{P}\dcal G_{\rm diff}$. For an abelian group $A$, $\text{char}(\pcal; A)$ denotes the set of characteristic classes with coefficients in $A$ for the class $\pcal$. Then, by \cite[Theorem 5.10]{CW17} and Proposition \ref{cor1.3}, we have the natural bijection
\[
{\rm char} (\mathsf{P}\dcal G_{\rm num}; A)\cong H^\ast(BG; A)
\]
and the retract diagram
\[
\begin{tikzpicture}
\node at (0,0)
{\begin{tikzcd}[column sep=large]
	{\rm char}(\mathsf{P}\dcal G_{\rm num}; A) \arrow[r, "ext"] & {\rm char}(\mathsf{P}\dcal G_{\rm diff}; A) \arrow[r, "res"] & {\rm char}(\mathsf{P}\dcal G_{\rm num}; A),
	\end{tikzcd}};
\draw[->] (-5,-0.4)--(-5,-1)--(5,-1)--(5,-0.4);
\node [right] at (0,-1.3) {1};
\end{tikzpicture}
\]
where $res$ is the obvious restriction map and $ext$ is the extension map introduced in Proposition \ref{cor1.3}.

We can also show that ${\rm char}(\mathsf{P}\dcal G; A)\cong {\rm char}(\mathsf{P}\dcal G_{\rm num}; A)$. To prove this, we define the map $ext: {\rm char}(\mathsf{P}\dcal G_{\rm num}; A) \longrightarrow {\rm char}(\mathsf{P}\dcal G;A)$ as follows. Let $\alpha(\cdot)$ be an element of ${\rm char}(\mathsf{P} \dcal G_{\rm num}; A)$ corresponding to $\alpha \in H^\ast(BG; A)$. For a given locally trivial principal $G$-bundle $\pi: P \longrightarrow X$, consider the $CW$-approximation $|S^\dcal(X)|_\dcal \xrightarrow{\ \ p_X\ \ } X$ in $\dcal$, which is the counit of the adjoint pair ($|\ |_\dcal, S^\dcal$) (see Remark \ref{modelstr}(2) and \cite[Section 3]{smh}). Since we can prove that every $CW$-complex in $\dcal$ is smoothly paracompact (see \cite{dRt}), the pullback $p_X^\ast P$ is a $\dcal$-numerable principal $G$-bundle. Thus, we can define the characteristic class $\alpha(P)$ of $P$ by $\alpha(P)=\alpha(p_X^\ast P)$ under the identification $H^\ast (X;A)\cong  H^\ast(|S^\dcal(X)|_\dcal, A)$. Then, it is obvious that the map $ext: {\rm char}(\mathsf{P} \dcal G_{\rm num}; A)\longrightarrow {\rm char}(\mathsf{P} \dcal G; A)$ and the obvious restriction map $res: {\rm char}(\mathsf{P}\dcal G; A)\longrightarrow {\rm char}(\mathsf{P}\dcal G_{\rm num}; A)$ are mutually inverses. We can easily see from Theorem \ref{bdletriviality} that the map $ext: {\rm char}(\mathsf{P}\dcal G_{\rm num}; A)\longrightarrow {\rm char}(\mathsf{P} \dcal G; A)$ is just the corestriction of $ext:{\rm char}(\mathsf{P}\dcal G_{\rm num}; A )\longrightarrow {\rm char}(\mathsf{P} \dcal G_{\rm diff}; A)$. (Recall that the class of locally trivial principal $G$-bundles also does not have the homotopy invariance property with respect to pullback and hence that it has no classifying space; see \cite[Section 3]{CW17}.)

We end this section by raising a problem on diffeological principal bundles.
\vspace{1ex}

\noindent Problem. Let $X$ be a $CW$-complex in $\dcal$ (or more generally, a cofibrant diffeological space); see \cite[Section 3.1]{smh}. Is every diffeological principal $G$-bundle over $X$ locally trivial?
\vspace{1ex} 

This problem asks whether there exists a non-locally-trivial diffeological principal bundle over a good diffeological space; all the non-locally-trivial diffeological principal bundles the author knows are ones over bad diffeological spaces.

If Problem is solved affirmatively, we can use the $CW$-approximation $|S^\dcal(X)|_\dcal \xrightarrow{\ \ p_X\ \ } X$ to directly construct the map
\[
{\rm char}(\mathsf{P} \dcal G_{\rm num}; A) \xrightarrow{\ \ ext \ \ } {\rm char}(\mathsf{P} \dcal G_{\rm diff}; A)
\]
which is the inverse of ${\rm char}(\mathsf{P}\dcal G_{\rm diff}; A) \xrightarrow{\ \ res\ \ } {\rm char}(\mathsf{P}\dcal G_{\rm num}; A)$.

Further, if Problem is solved affirmatively, then we can replace the singular functor $S^\dcal_{\rm aff}$ with $S^\dcal$ in Theorem \ref{bdletriviality} and Remark \ref{bdletriviality2}(1).

\begin{rem}\label{twoext}
	\begin{itemize}
		\item[{\rm (1)}] Results similar to those mentioned above hold in the category $\tcal$ of topological spaces. More precisely, the homotopy invariance property with respect to pullback need not hold for topological principal $G$-bundles which are not numerable, and hence the class of topological principal $G$-bundles does not have a classifying space (see \cite[Section 3]{CW17}, \cite{An}, and \cite{Go}). However, we have two ways of extending the characteristic class associated to a cohomology class $\alpha$ of the (topological) classifying space $BG$; one uses the $CW$-approximation $|S(X)|\overset{p_X}{\longrightarrow} X$ of the base and the other uses the theory of simplicial principal bundles. We can easily see that they define the same extension; the resulting map is denoted by
		\[
		{\rm char}(\mathsf{P}\tcal G_{\rm num}; A) \xrightarrow{\ \ ext\ \ } {\rm char}(\mathsf{P} \tcal G; A),
		\]
		where ${\rm char}(\mathsf{P}\tcal G_{\rm num}; A)$ and ${\rm char}(\mathsf{P} \tcal G; A)$ are defined in a way similar to the diffeological case. We then see that
		\[
		{\rm char} (\mathsf{P}\tcal G_{\rm num}; A) \cong H^\ast(BG;A)
		\]
		and that
		\[
		{\rm char}(\mathsf{P}\tcal G_{\rm num}; A) \overset{ext}{\underset{res}{\rightleftarrows}} {\rm char}(\mathsf{P}\tcal G; A)
		\]
		are mutually inverses.\vspace{-1.0ex}

\hspace{1em}The results here remain true even if $\tcal$ is replaced with the category $\czero$ of arc-generated spaces (see \cite[Proposition 5.14(1)]{smh}).
		\item[{\rm (2)}] Since the underlying topological space functor $\widetilde{\cdot}:\dcal \longrightarrow \ccal^0$ preserves finite products \cite[Proposition 2.13]{origin}, it induces the functor
$$ \widetilde{\cdot}: \mathsf{P}\dcal G \longrightarrow \mathsf{P}\ccal^{0}\widetilde{G} $$
(see \cite[Lemma 5.7 and Remark 5.8]{smh}). Thus, we use this functor to study the relation between characteristic classes of smooth principal $G$-bundles and ones of continuous principal $G$-bundles.\\
\hspace{1em}The natural inclusion $S^\dcal X \hookrightarrow S\widetilde{X}$ (see Section 2.3) induces the natural homomorphism
		\[
		H^\ast(X;A) \xleftarrow{\ \ \psi_X\ \ } H^\ast(\widetilde{X}; A),
		\]
		which along with \cite[Proposition 5.14]{smh}, defines the horizontal arrows in the following commutative diagram:
		\begin{center}
		\begin{tikzcd}
		H^\ast(BG; A) \arrow[d, "\cong" sloped] & H^\ast(B\widetilde{G}; A) \arrow[l, "\psi_{BG}"'] \arrow[d, "\cong"' sloped]\\
		{\rm char}(\mathsf{P}\dcal G_{\rm num}; A) \arrow[d, "ext"', "\cong" sloped] & {\rm char}(\mathsf{P}\czero \widetilde{G}_{\rm num}; A)\arrow[d, "ext", "\cong"' sloped] \arrow[l]\\
		{\rm char}(\mathsf{P}\dcal G; A) & {\rm char}(\mathsf{P}\czero \widetilde{G}; A). \arrow[l]
		\end{tikzcd}
		\end{center}
		We can easily see that the equality
		\[
		(\psi_{BG} \alpha)(P) = \psi_X(\alpha(\widetilde{P}))
		\]
		holds for $P\in \mathsf{P} \dcal G$. \\
\hspace{1em}If $G$ is a Lie group (or more generally, in the class $\vcal_\dcal$), then $H^\ast(BG; A)\xleftarrow{\psi_{BG}} H^\ast(B\widetilde{G}; A)$ is an isomorphism (see  \cite[Theorem 11.2, and Corollaries 1.6 and 5.16]{smh}), and hence all the arrows in the above commutative diagram are bijective. (Here, a Lie group is defined to be a group in the category $C^{\infty}$ of $C^{\infty}$-manifolds in the sence of \cite[Section 27]{KM}; see \cite[Section 2.2]{smh}.)
	\end{itemize}
\end{rem}

\section*{Acknowledgement}
In the conference ``Building-up Differentiable Homotopy Theory 2020" held at Shinshu University, Prof. J. D. Christensen gave an excellent talk on smooth classifying spaces. After the talk, Prof. Katsuhiko Kuribayashi asked me whether the ``only if'' part of Theorem \ref{bdletriviality}(1) holds and then we briefly discussed this subject with Prof. Christensen, which lead me to Proposition \ref{cor1.3}. I would like to show my greatest appreciation to them.

\end{document}